\documentclass[oneside]{amsart} 
\usepackage{amsmath}
\usepackage[inner=2.5cm,outer=25mm, top=25mm, bottom=25mm]{geometry}
\usepackage[latin1]{inputenc}
\usepackage{fancyhdr}
\usepackage{hyperref}
\usepackage{stmaryrd}
\usepackage{epsfig}
\usepackage{tikz}
\usepackage{euscript}
\usepackage{color}
\usepackage{pgfplots}
\usetikzlibrary{calc}       		
\usepackage[scaled=0.9]{helvet}		
\usepackage{courier}			
\usepackage{bbm}
\usepackage[latin1]{inputenc}
\usepackage{listings}
\usepackage{graphicx}
\usepackage{amsmath}
\usepackage{amssymb}
\usepackage{bbm}
\newtheorem{theorem}{Theorem}

\newtheorem{corollary}[theorem]{Corollary}
\newtheorem{definition}[theorem]{Definition}
\newtheorem{example}[theorem]{Example}

\newtheorem{lemma}[theorem]{Lemma}

\newtheorem{proposition}[theorem]{Proposition}
\newtheorem{remark}[theorem]{Remark}

\usepackage{listings}
\usepackage [all,arc,curve,color, arrow, matrix,frame]{xy}
\usepackage[english]{babel}
\usepackage{tikz}

\newcommand{\R}{{\mathbb R}}
\newcommand{\C}{{\mathbb C}}
\newcommand{\N}{{\mathbb N}}

\newcommand{\Real}{{\mbox{Re}}}

\newcommand{\spann}{{\mbox{span}}}
\newcommand{\ran}{{\mbox{ran}}}
\newcommand{\Imag}{{\mbox{Im}}}
\newcommand{\idty}{{\mathbbm{1}}}

\numberwithin{equation}{section}
\numberwithin{theorem}{section} 
\numberwithin{footnote}{section}
\begin{document}

\title[On non-proper dissipative extensions]{The Non-Proper Dissipative Extensions of a Dual Pair}

\author[C. Fischbacher]{Christoph Fischbacher$^1$}
\address{$^1$ Department of Mathematics, University of Alabama at Birmingham, Birmingham, AL 35294, USA}
\email{cfischb@uab.edu}

\date{}

%
\begin{abstract}
We consider dissipative operators $A$ of the form $A=S+iV$, where both $S$ and $V\geq 0$ are assumed to be symmetric but neither of them needs to be (essentially) selfadjoint. After a brief discussion of the relation of the operators $S\pm iV$ to dual pairs with the so called common core property, we present a necessary and sufficient condition for any extension of $A$ with domain contained in $\mathcal{D}((S-iV)^*)$ to be dissipative. We will discuss several special situations in which this condition can be expressed in a particularly nice form -- accessible to direct computations. Examples involving ordinary differential operators are given.
\end{abstract}
\numberwithin{equation}{section}
\numberwithin{theorem}{section} 
\numberwithin{footnote}{section}
\maketitle
\section{Introduction} 
In this note, we want to contribute towards the extension theory of dissipative operators $A$ of the form $A=S+iV$, where $S$ and $V\geq 0$ are both symmetric but neither of them needs to be selfadjoint or essentially selfadjoint.
In this sense, we will obtain a more general result than that of Crandall and Phillips \cite{Crandall}, who considered dissipative operators $A$ that were of the form\footnote{In \cite{Crandall}, a densely defined operator is called dissipative if its numerical range is confined to the left half plane $\Pi_-:=\{z\in\C: \Real(z)\leq 0\}$. Since we will call an operator dissipative if its numerical range is confined to the upper complex plane, we have changed the presentation of the results in \cite{Crandall} accordingly.} 
$A=S+iV$, 
where $S$ is symmetric, but $V\geq 0$ is assumed to be selfadjoint. However, we have stricter conditions on the domains of our extensions. 
\subsection{Dissipative operators and extension theory}
The study of non-selfadjoint operators has proven itself to be a very fruitful field of mathematical research. For an introduction into the many new phenomena and problems that arise if one gives up the condition of selfadjointness, we refer the interested reader to the classic monograph \cite{GK} and the references therein. We mention in particular the work of Brodskii and {Liv\v sic} who addressed questions such as the completeness of root vectors and introduced characteristic matrix functions and triangular models of non-selfadjoint operators \cite{BL58, Livs46, Livs54}.

In what follows, we will call a densely defined operator $A$ on a Hilbert space $\mathcal{H}$ \emph{dissipative} if and only if its numerical range is confined to the upper complex plane, i.e.\ if and only if 
$$\Imag\langle \psi,A\psi\rangle\geq 0$$
for any $\psi\in\mathcal{D}(A)$. Note that we have defined the sesquilinear form $\langle\cdot,\cdot\rangle$ to be antilinear in the first and linear in the second component. Moreover, we call a dissipative operator $A$ \emph{maximally dissipative} if it has no non-trivial dissipative operator extension, i.e. $A$ being maximally dissipative and $B$ being a dissipative operator extension of $A$ implies that $A=B$. Maximally dissipative operators possess various interesting features, e.g.\ they generate strongly continuous semigroups of contractions \cite{Phillips59} and always have a selfadjoint dilation \cite{NagyFoias}.

Thus, the theory of dissipative extensions of a given operator is an extensively studied problem (for an overview, we recommend the surveys \cite{Arli,AT2009} and all the references therein). Besides the classical results of von Neumann on the theory of selfadjoint extensions of a given symmetric operator \cite{vNeumann} and of Kre\u\i{n}, Birman, Vishik and Grubb on positive selfadjoint and maximally sectorial extensions of a given symmetric operator with positive numerical range \cite{Krein, Vishik, Birman, Grubb68, Alonso-Simon, Grubbfriendly}, let us also mention the results of authors like Arlinski\u\i , Belyi, Derkach, Kovalev, Malamud, Mogilevskii and {Tsekanovski\u\i} \cite{Arlinskii95, ABT2011, Kovalev, AT2005, DM91, DMT88, MM97, MM99, MM02, Tsek80, Tsek81} who have made many contributions using form methods and boundary triples in order to determine maximally sectorial and maximally accretive extensions of a given sectorial operator.\footnote{A densely defined operator $A$ is called (maximally) accretive if $(iA)$ is (maximally) dissipative. If in addition, there exists a $\phi\neq \pi/2$ such that $(e^{i\phi}A)$ is (maximally) dissipative as well, then it is called (maximally) sectorial.} Let us also mention examples, where explicit computations of maximally dissipative (resp.\ accretive) extensions for positive symmetric differential operators \cite{EvansKnowles85}, \cite{EvansKnowles86} and for sectorial Sturm-Liouville operators \cite{BrownEvans} have been made.

For the general problem of finding dissipative extensions of truly dissipative operators, Phillips showed that -- via the Cayley-transform and its inverse -- this is equivalent to finding contractive extensions of a non-densely defined contraction. This problem has been solved by Crandall \cite[Thm.\ 1 and Cor.\ 1]{CrandallContraction} who therefore has provided a full solution to the extension problem (note also the results in \cite{AG}). Crandall established that if $C$ is a contraction defined on a closed subspace $\mathcal{C}$ of a Hilbert space $\mathcal{H}$ and mapping to $\mathcal{H}$, all contractive extensions $\widetilde{C}$ of $C$ can be described via
\begin{equation*}
\widetilde{C}=CP_\mathcal{C}+(\idty-CP_\mathcal{C}(CP_\mathcal{C})^*)^{1/2}B(\idty-P_\mathcal{C})\:,
\end{equation*}
where $P_\mathcal{C}$ is the orthogonal projection onto $\mathcal{C}$ and $B$ is an arbitrary contraction on $\mathcal{H}$. However, for concrete applications, the operators involved in this construction are often very difficult to compute. Thus, in \cite{Crandall}, Crandall and Phillips made extra assumptions on the structure of the considered dissipative operator $A$ and required that it could be written in the form
\begin{equation} \label{eq:crandallrequire}
A=S+iV\:,
\end{equation}
where $S$ is symmetric, $V\geq 0$ is selfadjoint and $\mathcal{D}(A)=\mathcal{D}(S)=\mathcal{D}(V)$. Let us briefly describe their approach in the next section.
\subsection{The construction of Crandall and Phillips}

For the case \eqref{eq:crandallrequire} considered by Crandall and Phillips, it follows from non-negativity and selfadjointness of $V$ that the operator $(\idty+V)$ is a boundedly invertible bijection from $\mathcal{D}(V)$ onto $\mathcal{H}$. They then introduce the weighted Hilbert space $\mathcal{H}_{+1}$ which is the linear space $\mathcal{D}(V^{1/2})$ equipped with the inner product $\langle f,g\rangle_{+1}:=\langle (\idty+V)^{1/2}f,(\idty+V)^{1/2}g\rangle$. Using standard ideas of the construction of Gel'fand triples, they associate every element $f$ of $\mathcal{H}$ with an element $\ell_f$ of the dual space $\mathcal{H}_{+1}^*$ of $\mathcal{H}_{+1}$ via
\begin{equation*}
\ell_f(g):=\langle f,g\rangle\quad\text{for any}\quad g\in\mathcal{H}_{+1}\:,
\end{equation*}
which has norm equal to
\begin{align*}
\|\ell_f\|&=\|(\idty+V)^{-1/2}f\|=:\|f\|_{-1}\:.
\end{align*}
The space $\mathcal{H}_{-1}$ is then obtained as the completion of $\mathcal{H}$ in $\mathcal{H}_{+1}^*$ with respect to $\|\cdot\|_{-1}$. Since for any $f\in\mathcal{D}(V^{1/2})$ and for any $g\in\mathcal{H}$ we have that $\|f\|\leq\|(\idty+V)^{1/2}f\|=\|f\|_{+1}$ and $\|g\|\geq \|(\idty+V)^{-{1/2}}g\|=\|g\|_{-1}$, we obtain the following inclusions:
\begin{equation*}
\mathcal{H}_{+1}\subset\mathcal{H}\subset\mathcal{H}_{-1}\:.
\end{equation*}
In particular, this implies that $V$ is bounded as an operator from $\mathcal{H}_{+1}$ to $\mathcal{H}_{-1}$ --- a feature which Crandall and Phillips use in order to determine all maximally dissipative extensions of $A$ as an operator from $\mathcal{H}_{+1}$ to $\mathcal{H}_{-1}$ \cite[Thm.\ 1.1]{Crandall}. Having obtained a maximally dissipative operator $\widehat{A}$ from $\mathcal{H}_{+1}$ to $\mathcal{H}_{-1}$, they then construct a dissipative extension $\widehat{A}^0$ of $A$ (as an operator in $\mathcal{H}$) via
\begin{align*}
\widehat{A}^0:\quad\mathcal{D}(\widehat{A}^0)=\{f\in\mathcal{D}(\widehat{A}): \widehat{A}f\in\mathcal{H}\},\quad\widehat{A}^0f:=\widehat{A}f\:. 
\end{align*}
If $V$ is bounded, this provides a full characterization of all maximally dissipative extensions of $A$, since the spaces $\mathcal{H}_{+1},\mathcal{H}$ and $\mathcal{H}_{-1}$ are equivalent in this case. For the unbounded case, this construction yields dissipative extensions of $A$ that have domain contained in $\mathcal{D}(V^{1/2})$, which does not always provide a full description of all maximally dissipative extensions of $A$ (cf.\ \cite[Example 2]{Crandall}). Also, even if $\widehat{A}$ is a maximally dissipative operator from $\mathcal{H}_{+1}$ to $\mathcal{H}_{-1}$, it is possible that $\widehat{A}^0$ is not a maximally dissipative operator in $\mathcal{H}$ \cite[Example 1]{Crandall}. However, Crandall and Phillips prove a necessary and sufficient condition for when all maximally dissipative extensions $\widehat{A}$ from $\mathcal{H}_{+1}$ to $\mathcal{H}_{-1}$ induce also a maximally dissipative extension $\widehat{A}^0$ in $\mathcal{H}$ \cite[Thm. 3.3]{Crandall}.

\subsection{Our approach}
In a previous note \cite{FNW}, we considered so called dual pairs of operators $(A,\widetilde{A})$, where $A$ and $(-\widetilde{A})$ were assumed to be dissipative and to possess a common core, which means that there exists a linear space $\mathcal{D}\subset\mathcal{D}(A)\cap\mathcal{D}(\widetilde{A})$ such that $A=\overline{A\upharpoonright_\mathcal{D}}$ and $\widetilde{A}=\overline{\widetilde{A}\upharpoonright_\mathcal{D}}$. Given $A=S+iV$, we will define $\widetilde{A}:=S-iV$ and show that $(A,\widetilde{A})$ is such a dual pair (with common core). We then show that considering such dual pairs is an equivalent point of view to assuming that $A$ is of the form $A=S+iV$ (Lemma \ref{thm:ccore}). In this sense, our results are going to be an extension of \cite{FNW}, where we gave a criterion to determine whether an extension $\widehat{A}$ with the property that $A\subset\widehat{A}\subset\widetilde{A}^*$ is dissipative, since we will drop the requirement that $\widehat{A}\subset\widetilde{A}^*$, while keeping the condition that $\mathcal{D}(\widehat{A})\subset\mathcal{D}(\widetilde{A}^*)$, i.e.\ the restrictions on the domain of $\widehat{A}$ remain but the action of the extensions $\widehat{A}$ may differ from that of $\widetilde{A}^*$. 

As it turns out, the square-roots of the selfadjoint Friedrichs and Kre\u\i n-von Neumann extensions of $V$ -- denoted by $V_F^{1/2}$, respectively by $V_K^{1/2}$, will play an important part in the presentation of our main result (Theorem \ref{thm:alter}). In particular, we will single out three cases in which it will be possible to simplify the result of Theorem \ref{thm:alter} and express the necessary and sufficient condition for an extension $\widehat{A}$ to be dissipative in terms that only involve $V_F^{1/2}$ and $V_K^{1/2}$ in terms of the quadratic forms $\psi\mapsto\|V_*^{1/2}\psi\|^2$, where $*\in\{F,K\}$ -- a feature which makes it accessible to direct calculation. These three cases are given by (i) an additional restriction on the action of $\widehat{A}$, (ii) $V\geq \varepsilon > 0$, and (iii) $V$ has only one non-negative selfadjoint extension (i.e. $V_F=V_K$). We will also discuss the interplay between boundary conditions (determined by the choice of $\mathcal{D}(\widehat{A})$) and the ``deviation" of $\widehat{A}$ from being a ``proper" extension of $(A,\widetilde{A})$ -- (determined by $(\widetilde{A}^*-\widehat{A})\upharpoonright_{\mathcal{D}(\widehat{A})}$). We will show that there is a fundamentally different behavior between the case that $V$ has only one non-negative selfadjoint extension $V_F=V_K$ (Corollary \ref{coro:ausmachen}) and the case that $V_F\neq V_K$ (Example \ref{ex:1493}).
\section{Some definitions and previous results}
We start with a few basic definitions and results on dissipative operators.
Firstly, let us state a lemma on by how many linearly independent vectors the domain of a given closed dissipative operator with finite defect index has to increase in order to obtain a maximally dissipative extension.  
\begin{lemma}[Mentioned in \cite{Crandall}, see also \cite{thesis} for a proof.] \label{prop:dampfnudel}
Let $A$ be a closed and dissipative operator on a separable Hilbert space $\mathcal{H}$ such that $\dim\text{\emph{ker}}(A^*-i)<\infty$. Moreover, let $\widehat{A}$ be a dissipative extension of $A$. Then, $\widehat{A}$ is maximally dissipative if and only if 
\begin{equation*}
\dim \mathcal{D}(\widehat{A})/{\mathcal{D}(A)}=\dim\text{\emph{ker}}(A^*-i)\:.
\end{equation*}
\end{lemma} 
Next, let us introduce some convenient notation for complementary subspaces:
\begin{definition} Let $\mathcal{N,M}$ be (not necessarily closed) linear spaces such that $\mathcal{M}\subset\mathcal{N}$. With the notation $\mathcal{N}//\mathcal{M}$ we mean any subspace of $\mathcal{N}$, which is complementary to $\mathcal{M}$, i.e.
\begin{equation*}
(\mathcal{N}//\mathcal{M})+\mathcal{M}=\mathcal{N}\quad\text{and}\quad(\mathcal{N}//\mathcal{M})\cap\mathcal{M}=\{0\}\:.
\end{equation*}
\end{definition}
Finally, we will need the characterization of the Kre\u\i n-von Neumann extension of a given non-negative symmetric operator $V$, which has been shown by Ando and Nishio.
\begin{proposition}[{\cite[Thm.\ 1]{Ando-Nishio}}]
Let $V$ be a non-negative closed symmetric operator.  The selfadjoint and non-negative square root of the Kre\u\i n--von Neumann extension of $V$, which we denote by $V_K^{1/2}$, can be characterized as follows:
\begin{align*}
\mathcal{D}(V_K^{1/2})=\left\{h\in\mathcal{H}: \sup_{f\in\mathcal{D}(V):Vf\neq 0}\frac{|\langle h,Vf\rangle|^2}{\langle f,Vf\rangle}<\infty\right\}\:,\\
\text{for any}\:\: h\in\mathcal{D}(V_K^{1/2}):\quad\|V_K^{1/2}h\|^2= \sup_{f\in\mathcal{D}(V):Vf\neq 0}\frac{|\langle h,Vf\rangle|^2}{\langle f,Vf\rangle}\:.
\end{align*} \label{thm:nett}
\end{proposition}
For our purposes, it will be more convenient to use the following characterization of $\mathcal{D}(V_K^{1/2})$ and $\|V_K^{1/2}h\|$:
\begin{corollary} \label{coro:andonishio}
Let $V$ be a non-negative closed symmetric operator on a Hilbert space $\mathcal{H}$. Then, the square root of its Kre\u\i n--von Neumann extension can be characterized as follows 
\begin{align} \label{eq:mjunitsch}
\mathcal{D}(V_K^{1/2})=\left\{h\in\mathcal{H}: \sup_{g\in\text{\emph{ran}}(\widehat{V}^{1/2}\upharpoonright_{\mathcal{D}(V)}):\|g\|=1}{|\langle h,\widehat{V}^{1/2}g\rangle|}<\infty\right\}\:,\\
\text{for any}\:\: h\in\mathcal{D}(V_K^{1/2}):\quad\|V_K^{1/2}h\|^2=\sup_{g\in\text{\emph{ran}}(\widehat{V}^{1/2}\upharpoonright_{\mathcal{D}(V)}):\|g\|=1}{|\langle h,\widehat{V}^{1/2}g\rangle|^2}\:,
\end{align}
where $\widehat{V}$ is any non-negative selfadjoint extension of $V$.
\end{corollary}
\begin{proof}
Let us consider any $f\in\mathcal{D}(V)$ such that $Vf\neq 0$. Since $Vf=\widehat{V}f=\widehat{V}^{1/2}\widehat{V}^{1/2}f$, we then get
\begin{equation*}
\frac{|\langle h,Vf\rangle|^2}{\langle f,Vf\rangle}=\frac{|\langle h,\widehat{V}^{1/2}\widehat{V}^{1/2}f\rangle|^2}{\|\widehat{V}^{1/2}f\|^2}=\left|\left\langle h,\widehat{V}^{1/2}\left(\frac{\widehat{V}^{1/2}f}{\|\widehat{V}^{1/2}f\|}\right)\right\rangle\right|^2\:.
\end{equation*}
Now, observe that $\frac{\widehat{V}^{1/2}f}{\|\widehat{V}^{1/2}f\|}$ is a normalized element of $\ran(\widehat{V}^{1/2}\upharpoonright_{\mathcal{D}(V)})$. Conversely, for any normalized $g\in\ran(\widehat{V}^{1/2}\upharpoonright_{\mathcal{D}(V)})$, there exists a $f\in\mathcal{D}(V)$ with $Vf\neq 0$ such that $g=\frac{\widehat{V}^{1/2}f}{\|\widehat{V}^{1/2}f\|}$. This implies that
\begin{equation*}
\sup_{f\in\mathcal{D}(V):Vf\neq 0}\left|\left\langle h,\widehat{V}^{1/2}\left(\frac{\widehat{V}^{1/2}f}{\|\widehat{V}^{1/2}f\|}\right)\right\rangle\right|^2=\sup_{g\in\text{ran}(\widehat{V}^{1/2}\upharpoonright_{\mathcal{D}(V)}):\|g\|=1}|\langle h,\widehat{V}^{1/2}g\rangle|^2\:,
\end{equation*}
which --- together with Proposition \ref{thm:nett} --- yields the corollary.
\end{proof}
\section{The common core property} 
Given any bounded operator $A$, the decomposition into its selfadjoint real part $S:=(A+A^*)/2$ and selfadjoint imaginary part $V:=(A-A^*)/(2i)$ allows us to always write $A$ as $A=S+iV$. For the unbounded case, this is generally not possible as one has to be careful with the domains. However, in the case that it is possible to decompose $A$ as
\begin{equation} \label{eq:decomp}
A=S+iV\:,
\end{equation}
where both $S$ and $V\geq 0$ are symmetric and $\mathcal{D}(A)=\mathcal{D}(S)= \mathcal{D}(V)$, one can use the framework of dual pairs $(A,\widetilde{A})$ of operators to decompose $A$ analogously as in the bounded case. To this end, let us firstly recall their definition (see also \cite{Edmunds-Evans, LyantzeStorozh} for more details): 
\begin{definition} Let $(A,\widetilde{A})$ be a pair of densely defined and closable operators. We say that they form a {\bf{dual pair}} if 
\begin{equation*}
A\subset \widetilde{A}^*\quad\text{resp.}\quad \widetilde{A}\subset A^*\:.
\end{equation*}
In this case, $A$ is called a {\bf{formal adjoint}} of $\widetilde{A}$ and vice versa. Moreover, an operator $\widehat{A}$ such that $A\subset\widehat{A}\subset\widetilde{A}^*$ is called a {\bf{proper extension}} of the dual pair $(A,\widetilde{A})$.
\end{definition} 
It is then not hard to see that with the choice $\widetilde{A}:=S-iV$ ($\mathcal{D}(\widetilde{A})=\mathcal{D}(S)=\mathcal{D}(V)$), we have that $(A,\widetilde{A})$ is a dual pair since 
\begin{equation} \label{eq:kirchedorf}
\langle f,\widetilde{A}g\rangle=\langle f, (S-iV)g\rangle=\langle (S+iV)f,g\rangle=\langle Af,g\rangle
\end{equation}
for any $f\in\mathcal{D}(S+iV)$ and any $g\in\mathcal{D}(S-iV)$.

For the presentation of our results in \cite{FNW}, the notion of dual pairs $(A,\widetilde{A})$ satisfying the so called common core property was particularly useful. Let us restate the definition.
\begin{definition}
Let $(A,\widetilde{A})$ be a dual pair of closed operators. We say that it has the {\bf{common core property}} if there exists a subset $\mathcal{D}\subset\mathcal{D}(A)\cap\mathcal{D}(\widetilde{A})$ such that it is a core for $A$ as well as for $\widetilde{A}$: $$\overline{A\upharpoonright_{\mathcal{D}}}=A\quad\text{and}\quad\widetilde{A}=\overline{\widetilde{A}\upharpoonright_{\mathcal{D}}}\:.$$
\end{definition}
We are now prepared to show the link between dual pairs $(A,\widetilde{A})$ satisfying the common core property and dissipative operators $A$ that can be decomposed according to \eqref{eq:decomp}.
\begin{lemma} \label{thm:ccore} Let $(A,\widetilde{A})$ be a dual pair of closed operators satisfying the common core property with a common core $\mathcal{D}$, where $A$ is dissipative. Then there exist two symmetric operators $S$ and $V\geq 0$ with $\mathcal{D}=\mathcal{D}(S)=\mathcal{D}(V)$ such that
\begin{equation*}
A\upharpoonright_\mathcal{D}=S+iV\quad\text{and}\quad\widetilde{A}\upharpoonright_\mathcal{D}=S-iV\:.
\end{equation*}
Conversely, let $A$ be a dissipative operator of the form $A=S+iV$, where $S$ and $V\geq 0$ are symmetric operators and $\mathcal{D}(A)=\mathcal{D}(S)=\mathcal{D}(V)$. If we define $\widetilde{A}:=S-iV$, where $\mathcal{D}(\widetilde{A})=\mathcal{D}$, then their closures $(\overline{A},\overline{\widetilde{A}})$ form a dual pair that has the common core property. 
\end{lemma}
\begin{proof} If $(A,\widetilde{A})$ is a dual pair satisfying the common core condition, with $\mathcal{D}$ being a common core, we may define
\begin{equation}
S:=\frac{A+\widetilde{A}}{2}\upharpoonright_\mathcal{D}\quad\text{and}\quad V:=\frac{A-\widetilde{A}}{2i}\upharpoonright_\mathcal{D}\;.
\end{equation}
Firstly, observe that $A\upharpoonright_\mathcal{D}=S+iV$ and $\widetilde{A}\upharpoonright_\mathcal{D}=S-iV$. Next, let us show that $S$ and $V$ are symmetric and also that $V\geq 0$. To this end, let $\psi\in\mathcal{D}$ and consider
\begin{align*}
\langle \psi,S\psi\rangle&=\frac{1}{2}\langle \psi,(A+\widetilde{A})\psi\rangle=\frac{1}{2}(\langle \psi,A\psi\rangle+\langle \psi,A^*\psi\rangle)=\Real(\langle\psi,A\psi\rangle)\in\R\\
\langle \psi,V\psi\rangle&=\frac{1}{2i}\langle \psi,(A-\widetilde{A})\psi\rangle=\frac{1}{2i}(\langle \psi,A\psi\rangle-\langle \psi,A^*\psi\rangle)=\Imag(\langle\psi,A\psi\rangle)\geq 0\:,
\end{align*} 
where the inequality follows from dissipativity of $A$.
Now, let $A$ be a dissipative operator of the form $A=S+iV$, where $S$ and $V\geq 0$ are symmetric and $\mathcal{D}(A)=\mathcal{D}(S)=\mathcal{D}(V)$. In \eqref{eq:kirchedorf}, we have already shown that $A$ and $\widetilde{A}:=S-iV$ form a dual pair and so do their closures $(\overline{A},\overline{\widetilde{A}})$, which therefore is a dual pair that has the common core property with common core $\mathcal{D}=\mathcal{D}(S)=\mathcal{D}(V)$.
\end{proof}
\section{The main result}
We are now prepared to prove our main result. Before we proceed, we need to show the following two lemmas:
\begin{lemma} \label{lemma:friedrichsdense}
Let $V$ be a non-negative symmetric operator. Then
\\

i) $\text{\emph{ran}}(V_F^{1/2}\upharpoonright_{\mathcal{D}(V)})$ is dense in $\overline{\text{\emph{ran}}(V_F^{1/2})}$

ii) $\text{\emph{ran}}(V_K^{1/2}\upharpoonright_{\mathcal{D}(V)})$ is dense in $\overline{\text{\emph{ran}}(V_K^{1/2})}$. 
\end{lemma}
\begin{proof}
i) By construction of the Friedrichs extension, we know that for any $\psi\in\mathcal{D}(V_F^{1/2})$, there exists a sequence $\{\psi_n\}\subset\mathcal{D}(V)$, such that 
\begin{equation*}
\lim_{n\rightarrow\infty}(\|\psi-\psi_n\|^2+\|V_F^{1/2}(\psi-\psi_n)\|^2)=0\:,
\end{equation*}
which implies in particular that $\lim_{n\rightarrow\infty} V_F^{1/2}\psi_n=V_F^{1/2}\psi$, i.e. $\ran(V_F^{1/2})\subset\overline{\ran(V_F^{1/2}\upharpoonright_{\mathcal{D}(V)})}$. On the other hand, since $\ran(V_F^{1/2}\upharpoonright_{\mathcal{D}(V)})\subset\ran(V_F^{1/2})$, the assertion follows from taking closures.

ii) Any element of $\ran(V_K^{1/2})$ is of the form $V_K^{1/2}h$, where $h\in\mathcal{D}(V_K^{1/2})$. By Corollary \ref{coro:andonishio} with the choice $\widehat{V}=V_K$, we have that
\begin{align} \label{eq:jeremy}
\|V_K^{1/2}h\|^2=\sup\left\{|\langle V_K^{1/2}h,g\rangle|^2, g\in\ran(V_K^{1/2}\upharpoonright_{\mathcal{D}(V)}):\|g\|=1\right\}\:.
\end{align}
But this implies that $\ran(V_K^{1/2}\upharpoonright_{\mathcal{D}(V)})$ is dense in $\overline{\ran(V_K^{1/2})}$. To see why, assume that there exists a $\varphi\in \overline{\ran(V_K^{1/2})}$ such that $\|\varphi\|=1$ and
$\langle \varphi,g\rangle=0$ for all $g\in\ran(V_K^{1/2}\upharpoonright_{\mathcal{D}(V)})$. Take a  $V_K^{1/2}h\in\ran(V_K^{1/2})$, with $\|V_K^{1/2}h\|=1$ such that $\|V_K^{1/2}h-\varphi\|^2<\varepsilon$ for some $0<\varepsilon<1$ small enough. Then, for any $g\in\ran(V_K^{1/2}\upharpoonright_{\mathcal{D}(V)})$, we get
\begin{equation*}
|\langle V_K^{1/2}h,g\rangle|^2=|\langle V_K^{1/2}h-\varphi,g\rangle|^2\leq \|V_K^{1/2}h-\varphi\|^2\|g\|^2\leq \varepsilon\|g\|^2\:.
\end{equation*}
Taking the supremum over all $g\in\ran(V_K^{1/2}\upharpoonright_{\mathcal{D}(V)})$ with $\|g\|=1$, we arrive at a contradiction, since the supremum of the left hand side is $1$ whereas the supremum of the right hand side is $\varepsilon<1$. This shows the lemma.
\end{proof}

\begin{lemma} \label{lemma:frischerfisch} Let $V$ be a non-negative symmetric operator and let $V_F$ and $V_K$ denote its Friedrichs, resp. its Kre\u\i n-von Neumann extension. Then there exists a partial isometry $\mathcal{U}$ on $\mathcal{H}$ such that 
\begin{equation} \label{eq:kuckuck}
V_K^{1/2}h=\mathcal{U}{V}_F^{1/2}h
\end{equation}
for all $h\in\mathcal{D}(V_F^{1/2})$. The map $\mathcal{U}$ is an isometry on $\overline{\ran(V_F^{1/2})}$ and its range $\ran(\mathcal{U})$ is contained in $\overline{\ran(V_K^{1/2})}$.
\end{lemma}
\begin{proof} Since we have that $V_K\leq V_F$, it is clear that $\mathcal{D}(V_F^{1/2})\subset\mathcal{D}(V_K^{1/2})$. Moreover, by Proposition \ref{thm:nett}, for any $h\in\mathcal{D}(V_F^{1/2})\subset\mathcal{D}(V_K^{1/2})$, we have that 
\begin{align*}
\|V_K^{1/2}h\|^2&=\sup_{f\in\mathcal{D}(V):Vf\neq 0}\frac{|\langle h,Vf\rangle|^2}{\langle f,Vf\rangle}=\sup_{f\in\mathcal{D}(V):Vf\neq 0}\frac{|\langle h,V_F^{1/2}V_F^{1/2}f\rangle|^2}{\langle f,V_F^{1/2}V_F^{1/2}f\rangle}\\&=\sup_{f\in\mathcal{D}(V):Vf\neq 0}\frac{|\langle V_F^{1/2} h,V_F^{1/2}f\rangle|^2}{\|V_F^{1/2}f\|^2}=\|V_F^{1/2}h\|^2\:,
\end{align*}
where we have used that $\ran (V_F^{1/2}\upharpoonright_{\mathcal{D}(V)})$ is dense in $\overline{\ran(V_F^{1/2})}$ by Lemma  \ref{lemma:friedrichsdense}.
This implies that the linear map
\begin{align*}
\mathcal{U}_0: \quad \ran(V_F^{1/2})&\rightarrow\ran\left(V_K^{1/2}\upharpoonright_{\mathcal{D}(V_F^{1/2})}\right)\\
V_F^{1/2}h&\mapsto V_K^{1/2}h
\end{align*}
is isometric. Since, trivially, $\ran(V_F^{1/2})$ is dense in $\overline{\ran(V_F^{1/2})}$, there exists a unique isometric extension $\mathcal{U}_0\subset\mathcal{U}$ on $\overline{\ran(V_F^{1/2})}$. Setting
$\mathcal{U}k=0$ for all $k\in\ker(V_F^{1/2})=\ran(V_F^{1/2})^\perp$ defines $\mathcal{U}$ as a partial isometry on the whole Hilbert space $\mathcal{H}$. Here, $\mathcal{M}^\perp$ denotes the orthogonal complement of a linear space $\mathcal{M}$. Moreover, since $$\ran(\mathcal{U}_0)=\ran\left(V_K^{1/2}\upharpoonright_{\mathcal{D}(V_F^{1/2})}\right)\subset\ran(V_K^{1/2})\:,$$ this implies that $\ran(\mathcal{U})$ is contained in $\overline{\ran(V_K^{1/2})}$ and thus the lemma.
\end{proof}

Given a dual pair $(A,\widetilde{A})$, let us introduce the following convenient way of parametrizing all extensions of $A$ which have domain contained in $\mathcal{D}(\widetilde{A}^*)$:
\begin{definition} \label{def:quantz}
Let $(A,\widetilde{A})$ be a dual pair, where $A$ is dissipative and $\widetilde{A}$ is antidissipative. Let $\mathcal{V}\subset\mathcal{D}(\widetilde{A}^*)//\mathcal{D}(A)$ be a linear space, which means that $\mathcal{V}\subset\mathcal{D}(\widetilde{A}^*)$ and $\mathcal{V}\cap\mathcal{D}(A)=\{0\}$. Moreover, let $\mathcal{L}$ be a linear operator from $\mathcal{V}$ into $\mathcal{H}$. Then, the operator $A_{\mathcal{V},\mathcal{L}}$ is given by
\begin{align*}
A_{\mathcal{V},\mathcal{L}}:\qquad\mathcal{D}(A_{\mathcal{V},\mathcal{L}})&=\mathcal{D}(A)\dot{+}\mathcal{V}\\
(f+v)&\mapsto \widetilde{A}^*(f+v)+\mathcal{L}v\:,
\end{align*}
where $f\in\mathcal{D}(A)$ and $v\in\mathcal{V}$. If $\mathcal{L}$ is the zero-operator, i.e. $\mathcal{L}=0$, we define $A_{\mathcal{V},0}=:A_\mathcal{V}$.
\end{definition}
Note that the operator $\mathcal{L}$ can be interpreted as the deviation of $A_{\mathcal{L,V}}$ from $\widetilde{A}^*$, since for any $v\in\mathcal{V}$, we get that
\begin{equation*}
(\widetilde{A}^*-A_{\mathcal{V,L}})v=\mathcal{L}v\:.
\end{equation*}
Let us now show the main theorem:
\begin{theorem} \label{thm:alter} Let $(A,\widetilde{A})$ be a dual pair that has the common core property, where $A$ is dissipative.
Moreover, assume that 

$$v\in\mathcal{D}(V_K^{1/2}) \qquad\text{and}\qquad \mathcal{L}v\in \text{\emph{ran}}(V_F^{1/2})=\mathcal{D}(V_F^{-1/2})$$

for all $v\in\mathcal{V}$.
Then, $A_{\mathcal{V},\mathcal{L}}$ is dissipative if and only if for all $v\in\mathcal{V}$ we have
\begin{equation} \label{eq:gstinkert}
\Imag\langle v,(\widetilde{A}^*+\mathcal{L})v\rangle\geq\frac{1}{4}\|\mathcal{U}V_F^{-1/2}\mathcal{L}v+2iV_K^{1/2}v\|^2\:.
\end{equation}
Here, $V^{-1/2}_F$ denotes the inverse of $V_F^{1/2}$ as an operator in $\overline{\text{\emph{ran}}(V_F^{1/2})}$, which is given by
\begin{align}
V_F^{-1/2}: \qquad \mathcal{D}(V_F^{-1/2})=\text{\emph{ran}} V_F^{1/2}&\rightarrow\mathcal{D}(V_F^{1/2})\cap\overline{\text{\emph{ran}}(V_F^{1/2})}\notag\\
V_F^{1/2}f&\mapsto f\:,
\label{eq:krakau}
\end{align}
which a well-defined non-negative selfadjoint operator on the Hilbert space $\overline{\ran(V_F^{1/2})}$.
The operator $\mathcal{U}$ is the partial isometry as defined in Lemma \ref{lemma:frischerfisch}.
\end{theorem}
\begin{proof}
Let us start be showing that the above conditions are sufficient. To this end, let $\mathcal{D}\subset\mathcal{D}(A)\cap\mathcal{D}(\widetilde{A})$ denote a common core for $A$ and $\widetilde{A}$. For any $f\in\mathcal{D}$ and any $v\in\mathcal{V}$, we then get
\begin{align*}
\Imag\langle f&+v,A_{\mathcal{V},\mathcal{L}}(f+v)\rangle=\Imag\langle f+v,\widetilde{{A}}^*(f+v)\rangle+\Imag \langle f+v,\mathcal{L}v\rangle\\&=\langle f,Vf\rangle+\Imag\langle v,2iVf\rangle+\Imag\langle v,(\widetilde{A}^*+\mathcal{L})v\rangle+\Imag\langle f,\mathcal{L}v\rangle\\
&=\|V_K^{1/2}f\|^2+\Imag\langle v,2iV_K^{1/2}V_K^{1/2}f\rangle+\Imag\langle v,(\widetilde{A}^*+\mathcal{L})v\rangle+\Imag\langle V_F^{-1/2}V_F^{1/2}f,\mathcal{L}v\rangle\\
&=\|V_K^{1/2}f\|^2+\Imag\langle V_K^{1/2}v,2iV_K^{1/2}f\rangle+\Imag\langle v,(\widetilde{A}^*+\mathcal{L})v\rangle+\Imag\langle \mathcal{U}V_F^{1/2}f, \mathcal{U}V_F^{-1/2}\mathcal{L}v\rangle\\
&=\|V_K^{1/2}f\|^2+\Imag\langle V_K^{1/2}v,2iV_K^{1/2}f\rangle+\Imag\langle v,(\widetilde{A}^*+\mathcal{L})v\rangle+\Imag\langle V_K^{1/2}f, \mathcal{U}V_F^{-1/2}\mathcal{L}v\rangle\\
&=\|V_K^{1/2}f\|^2+\Imag\langle v,(\widetilde{A}^*+\mathcal{L})v\rangle+\Imag\langle V_K^{1/2}f,(\mathcal{U}V_F^{-1/2}\mathcal{L}+2iV_K^{1/2})v\rangle\\
&\geq \|V_K^{1/2}f\|^2+\frac{1}{4}\|\mathcal{U}V_F^{-1/2}\mathcal{L}v+2iV_K^{1/2}v\|^2+\Imag\langle V_K^{1/2}f,(\mathcal{U}V_F^{-1/2}\mathcal{L}+2iV_K^{1/2})v\rangle\\\
&\geq \|V_K^{1/2}f\|^2+\frac{1}{4}\|\mathcal{U}V_F^{-1/2}\mathcal{L}v+2iV_K^{1/2}v\|^2-\|V_K^{1/2}f\| \|(\mathcal{U}V_F^{-1/2}\mathcal{L}+2iV_K^{1/2})v\|\\
&=\left(\|V_K^{1/2}f\|-\frac{1}{2}\|\mathcal{U}V_F^{-1/2}\mathcal{L}v+2iV_K^{1/2}v\|\right)^2\geq 0\:.
\end{align*}
Let us now show that Condition \eqref{eq:gstinkert} is also necessary. Assume that it is not satisfied, i.e. that there exists a $v\in\mathcal{V}$ such that 
\begin{equation} \label{eq:schlechti}
\Imag\langle v,(\widetilde{A}^*+\mathcal{L})v\rangle-\frac{1}{4}\|\mathcal{U}V_F^{-1/2}\mathcal{L}v+2iV_K^{1/2}v\|^2\leq -\varepsilon
\end{equation}
for some $\varepsilon>0$. By Lemma \ref{lemma:frischerfisch}, we have that
$(\mathcal{U}V_F^{-1/2}\mathcal{L}v+2iV_K^{1/2}v)\in\overline{\ran(V_K^{1/2})}$. Thus, by Lemma \ref{lemma:friedrichsdense} ii), there exists a sequence $\{f_n\}\subset{\mathcal{D}(V)}$ such that
$$ V_K^{1/2}f_n\overset{n\rightarrow\infty}{\longrightarrow}\frac{-i}{2}(\mathcal{U}V_F^{-1/2}\mathcal{L}v+2iV_K^{1/2}v)\:,
$$
which means by \eqref{eq:schlechti} that 
\begin{align*}
&\Imag\langle f_n+v,A_{\mathcal{V,L}}(f_n+v)\rangle\\&=\|V_K^{1/2}f_n\|^2+\Imag\langle v,(\widetilde{A}^*+\mathcal{L})v\rangle+\Imag\langle V_K^{1/2}f_n,\mathcal{U}V_F^{-1/2}\mathcal{L}v+2iV_K^{1/2}v\rangle \overset{n\rightarrow\infty}{\longrightarrow}-\varepsilon<0\:,
\end{align*}
which means that $A_{\mathcal{V,L}}$ is not dissipative in this case. This shows the theorem.
\end{proof}
Note that for the proof of Theorem \ref{thm:alter}, we have assumed that 
$$v\in\mathcal{D}(V_K^{1/2}) \qquad\text{and}\qquad \mathcal{L}v\in \text{\emph{ran}}(V_F^{1/2})=\mathcal{D}(V_F^{-1/2})\:.$$
Let us now show that given either condition, the other is necessary for (we will comment on the case that neither condition is satisfied after the proof of the following theorem).

\begin{theorem} Let $(A,\widetilde{A})$ be a dual pair satisfying the common core condition, where $A$ is dissipative. 

i) If $\ran(\mathcal{L})\subset\ran(V_F^{1/2})$, then it is necessary that $\mathcal{V}\subset\mathcal{D}(V_K^{1/2})$ for $A_{\mathcal{V,L}}$ to be dissipative.

ii) If $\mathcal{V}\subset\mathcal{D}(V_K^{1/2})$, then it is necessary that $\ran(\mathcal{L})\subset\ran(V_F^{1/2})$ for $A_{\mathcal{V,L}}$ to be dissipative.
\label{thm:kuckuckswalzer}
\end{theorem}
\begin{proof}

i) If $\ran(\mathcal{L})\subset\ran(V_F^{1/2})$, this means that for any $v\in\mathcal{V}$ there exists a $\phi_v\in\mathcal{D}(V_F^{1/2})$ such that $\mathcal{L}v=V_F^{1/2}\phi_v$. Thus, for any $f\in\mathcal{D}$ we can write 
\begin{equation} \label{eq:trompete}
\Imag\langle f+v,A_{\mathcal{V,L}}(f+v)\rangle= \|V_F^{1/2}f\|^2+\Imag\langle v,2iV_F^{1/2}V_F^{1/2}f\rangle+\Imag\langle v,(\widetilde{A}^*+\mathcal{L})v\rangle-\Imag\langle \phi_v,V_F^{1/2}f\rangle\:.
\end{equation}
Now, assume that there exists a $v\in\mathcal{V}$ such that $v\notin\mathcal{D}(V_K^{1/2})$. By Corollary \ref{coro:andonishio} with the choice $\widehat{V}=V_F$ and Lemma \ref{lemma:friedrichsdense} i), this means that there exists a sequence $\{f_n\}\subset\mathcal{D}(V)$ with $\|V_F^{1/2}f_n\|=1$ for any $n$ and a sequence of complex phases $\{e^{i\varphi_n}\}$ such that 
\begin{equation*}
\lim_{n\rightarrow\infty}\Imag\langle v,2ie^{i\varphi_n}V_F^{1/2}V_F^{1/2}f_n\rangle=-\lim_{n\rightarrow\infty}\left|\langle v,2V_F^{1/2}V_F^{1/2}f_n\rangle\right|=-\infty\:.
\end{equation*}
Since all other terms in \eqref{eq:trompete} stay bounded, this shows that $A_{\mathcal{V,L}}$ cannot be dissipative in this case.\\
ii) We start by showing that in this case, it is necessary that $\mathcal{L}v\perp\ker V_F^{1/2}$ for all $v\in\mathcal{V}$. Assume this is not the case, i.e. that there exists a $v\in\mathcal{V}$ and a $k\in\ker (V_F^{1/2})=\ker (V_F)$ such that  $\langle \mathcal{L}v,k\rangle\neq 0$. Without loss of generality we may assume that $\Imag\langle \mathcal{L}v,k\rangle=1$. Now, since $\mathcal{D}(V)$ is a core for $V_F^{1/2}$, we can pick a sequence $\{f_n\}\subset\mathcal{D}(V)$ such that $f_n\rightarrow \lambda k$ and $V_F^{1/2}f_n\rightarrow \lambda V_F^{1/2}k=0$, where $\lambda\in\C$ is an arbitrary complex number. We then get
\begin{align*}
&\lim_{n\rightarrow\infty}\Imag\langle (f_n+v,A_{\mathcal{V,L}}(f_n+v)\rangle\\=&\lim_{n\rightarrow\infty}\left(\|V_F^{1/2}f_n\|^2+\Imag\langle v,2iV_K^{1/2}V_K^{1/2}f_n\rangle+\Imag\langle v,(\widetilde{A}^*+\mathcal{L})v\rangle-\Imag\langle \mathcal{L}v,f_n\rangle\right)\\
\overset{\eqref{eq:kuckuck}}{=}&\lim_{n\rightarrow\infty}\left(\|V_F^{1/2}f_n\|^2+\Imag\langle V_K^{1/2}v,2i\mathcal{U}V_F^{1/2}f_n\rangle+\Imag\langle v,(\widetilde{A}^*+\mathcal{L})v\rangle-\Imag\langle \mathcal{L}v,f_n\rangle\right)\\
=&\lim_{n\rightarrow\infty}\left(\|V_F^{1/2}f_n\|^2+\Imag\langle \mathcal{U}^*V_K^{1/2}v,2iV_F^{1/2}f_n\rangle+\Imag\langle v,(\widetilde{A}^*+\mathcal{L})v\rangle-\Imag\langle \mathcal{L}v,f_n\rangle\right)\\
=&\,\Imag\langle v,(\widetilde{A}^*+\mathcal{L})v\rangle-\Imag\lambda\:,
\end{align*}
which is negative if we choose $\Imag\lambda$ large enough. This contradicts the dissipativity of $A_{\mathcal{V,L}}$. Hence $\ran(\mathcal{L})\subset(\ker V_F^{1/2})^\perp=\overline{\ran(V_F^{1/2})}$. Now, since $\ker V_F^{1/2}$ is a reducing subspace for $V_F^{1/2}$, we have that the operator $V_F^{-1/2}$ is a well-defined non-negative selfadjoint operator on the Hilbert space $\overline{\ran(V_F^{1/2})}$. Also, note that $\overline{\ran(V_F^{1/2})}$ reduces $V_F^{1/2}$. Now, assume that there is a $v\in\mathcal{V}$, such that
 $\mathcal{L}v\notin\ran(V_F^{1/2})=\mathcal{D}(V_F^{-1/2})$. This means that we can pick a sequence $\{f_{n}\}\subset{\mathcal{D}(V)}$, where $\|V_F^{1/2}f_{n}\|=1$ for all $n$, such that $$\lim_{n\rightarrow\infty}\Imag\langle \mathcal{L}v,V_F^{-1/2}V_F^{1/2}f_n\rangle=+\infty\:,$$
since otherwise the map $g\mapsto\langle \mathcal{L}v,V_F^{-1/2}g\rangle$ would be a bounded linear functional on ${\ran(V_F^{1/2}\upharpoonright_{\mathcal{D}(V)})}$, which -- by Lemma \ref{lemma:friedrichsdense} (i) -- is dense in $\overline{\ran(V_F^{1/2}})$ --- a contradiction to $\mathcal{L}v\notin\mathcal{D}(V_F^{-1/2})$. 
Thus, we get
\begin{align*}
&\Imag\langle (f_n+v,A_{\mathcal{V,L}}(f_n+v)\rangle\\
=&\|V_F^{1/2}f_n\|^2+\Imag\langle \mathcal{U}^*V_K^{1/2}v,2iV_F^{1/2}f_n\rangle+\Imag\langle v,(\widetilde{A}^*+\mathcal{L})v\rangle-\Imag\langle \mathcal{L}v,f_n\rangle\\
\leq &  1+2\|\mathcal{U}^*V_K^{1/2}v\|+\Imag\langle v,(\widetilde{A}^*+\mathcal{L})v\rangle-\Imag\langle \mathcal{L}v,V_F^{-1/2}V_F^{1/2}f_n\rangle\overset{n\rightarrow\infty}{\longrightarrow}-\infty\:,
\end{align*}
which means that $A_{\mathcal{V,L}}$ cannot be dissipative in this case either. This finishes the proof.
\end{proof}
\begin{remark} For the case of proper extensions, i.e. for $\mathcal{L}=0$, Theorems \ref{thm:alter} and \ref{thm:kuckuckswalzer} readily imply our previous result \cite[Thm. 4.7]{FNW}.
\end{remark}
\begin{remark} \normalfont Since for any $f_n\in\mathcal{D}(A)$ and $v\in\mathcal{V}$ we get
\begin{equation} \label{eq:pachelbel}
\Imag\langle f_n+v,A_{\mathcal{V,L}}(f_n+v)\rangle=\Imag\langle v,(\widetilde{A}^*+\mathcal{L})v\rangle+\|V_K^{1/2}f_n\|^2+\Imag\langle v,2iV_K^{1/2}V_K^{1/2}f_n\rangle+\Imag\langle V_F^{-1/2}V_F^{1/2}f_n,\mathcal{L}v\rangle\:,
\end{equation}
the condition $\mathcal{V}\subset\mathcal{D}(V_K^{1/2})$ controls the term $\Imag\langle v,2iV_K^{1/2}V_K^{1/2}f_n\rangle$ while the condition $\ran(\mathcal{L})\subset\ran(V_F^{1/2})$ ensures that the term $\Imag\langle V_F^{-1/2}V_F^{1/2}f_n,\mathcal{L}v\rangle$ can be controlled when minimizing $\Imag\langle f_n+v,A_{\mathcal{V,L}}(f_n+v)\rangle$. 

But in the case that neither condition is satisfied it could happen that the last term in \eqref{eq:pachelbel} does not stay bounded either and instead ``competes" against the $\Imag\langle v,2iV_K^{1/2}V_K^{1/2}f_n\rangle$ that would go to $-\infty$ for a suitable choice of a sequence $\{f_n\}$. Thus, in the situation $v\notin\mathcal{D}(V_K^{1/2})$ and $\mathcal{L}v\notin\ran (V_F^{1/2})$ it is not clear whether it is in general possible that $A_{\mathcal{V,L}}$ is dissipative. Moreover, since it is difficult to compute $V_F^{1/2}$, $V_F^{-1/2}$ and $V_K^{1/2}$ explicitly, we were not able to construct such an example. (The elementary case of $V$ being a multiplication operator or --- more generally --- an essentially selfadjoint operator will be discussed in Lemma \ref{lemma:singulariter}.)

\end{remark}
\section{Applications of the main theorem}
Despite its theoretical merit, Theorem \ref{thm:alter} does not seem to be very useful for practical applications, since it is in general very difficult to explicitly compute the square-roots $V_K^{1/2}$ and $V_F^{-1/2}$ as they occur in the statement of the theorem. Moreover, we do not have explicit knowledge of the partial isometry $\mathcal{U}$. In this section, we are therefore going to single out three situations in which Condition \eqref{eq:gstinkert} can be simplified and made accessible to direct computations.
\subsection{An additional restriction on $\ran(\mathcal{L})$} \label{subsec:additionalreq} For the statement of Theorem \ref{thm:alter}, we have assumed that $\ran(\mathcal{L})\subset\ran(V_F^{1/2})$. If we make the even stricter assumption that $\ran(\mathcal{L})\subset\ran(V_F)$, we can simplify the result of Theorem \ref{thm:alter}:
\begin{corollary} \label{coro:wolfratshausen} Let $(A,\widetilde{A})$ be dual pair satisfying the common core property, where $A$ is dissipative. Moreover, assume that $\ran(\mathcal{L})\subset\ran(V_F)$. In this case, we write $\mathcal{L}v=V_F\phi_v$, where $\phi_v\in\mathcal{D}(V_F)$, which is determined up to elements in $\ker(V_F)$. Then, $A_\mathcal{V,L}$ is dissipative if and only if $\mathcal{V}\subset\mathcal{D}(V_K^{1/2})$ and for all $v\in\mathcal{V}$, we have that
\begin{equation} \label{eq:scan3}
\Imag\langle v,\widetilde{A}^*v\rangle+\Imag\langle v,V_F\phi_v\rangle\geq\frac{1}{4}\|V_K^{1/2}(\phi_v+2iv)\|^2\:.
\end{equation}
\end{corollary}
\begin{proof} Since $\ran(\mathcal{L})\subset\ran(V_F)\subset\ran(V_F^{1/2})$ by assumption, it follows from Theorem \ref{thm:kuckuckswalzer} i), that it is necessary that $\mathcal{V}\subset\mathcal{D}(V_K^{1/2})$ for $A_\mathcal{V,L}$ to be dissipative. Condition \eqref{eq:scan3} follows from \eqref{eq:gstinkert}, where we substitute $\mathcal{L}v=V_F\phi_v$ to get
\begin{equation*}
\Imag\langle v,\widetilde{A}^*v\rangle+\Imag\langle v,V_F\phi_v\rangle\geq\frac{1}{4}\|\mathcal{U}V^{-1/2}_FV_F\phi_v+2iV_K^{1/2}v\|^2=\frac{1}{4}\|V_K^{1/2}(\phi_v+2iv)\|^2\:,
\end{equation*}
which is the desired result.
\end{proof}
\begin{example} \normalfont \label{ex:potsdam}
Let $\mathcal{H}=L^2(0,\infty)$, assume that the real potential $W\in L^2(0,\infty)$ and consider the dual pair of closed operators $(A,\widetilde{A})$ given by
\begin{align*}
A:\qquad\mathcal{D}(A)=H^2_0(0,\infty),\quad (Af)(x)&= -if''(x)+W(x)f(x)\\
\widetilde{A}:\qquad\mathcal{D}(\widetilde{A})=H^2_0(0,\infty),\quad (\widetilde{A}f)(x)&=if''(x)+W(x)f(x)\:,
\end{align*}
which has the common core property since $\mathcal{D}(A)=\mathcal{D}(\widetilde{A})$ and $(A,\widetilde{A})$ are closed.
Their adjoints are given by
\begin{align*}
\widetilde{A}^*:\qquad\mathcal{D}(\widetilde{A}^*)=H^2(0,\infty),\quad (\widetilde{A}^*f)(x)&= -if''(x)+W(x)f(x)\\
A^*:\qquad\mathcal{D}(A^*)=H^2(0,\infty),\quad (A^*f)(x)&= if''(x)+W(x)f(x)\:.
\end{align*}
Moreover, the ``imaginary part" $V$ and its adjoint $V^*$ are given by 
\begin{align*}
V:\qquad\mathcal{D}(V)&=H^2_0(0,\infty),\quad f\mapsto -f''\\
V^*:\qquad\mathcal{D}(V^*)&=H^2(0,\infty),\quad f\mapsto -f''\:.
\end{align*}
Since $$\ker(V^*\pm i)=\spann\left\{\exp\left(-\frac{1\pm i}{\sqrt{2}}x\right)\right\}\:,$$ and $$\mathcal{D}(\widetilde{A}^*)=\mathcal{D}(V^*)=\mathcal{D}(V)\dot{+}\ker(V^*+i)\dot{+}\ker(V^*-i)=\mathcal{D}(A)\dot{+}\ker(V^*+i)\dot{+}\ker(V^*-i)\:,$$ we may choose
\begin{align*}
\mathcal{D}(V^*)//\mathcal{D}(V)&=\mathcal{D}(\widetilde{A}^*)//\mathcal{D}(A)\\&=\spann\left\{\exp\left(-\frac{1+ i}{\sqrt{2}}x\right),\exp\left(-\frac{1- i}{\sqrt{2}}x\right)\right\}=\spann\{\sigma,\tau\}\:.
\end{align*}
The functions $\sigma$ and $\tau$ are suitable linear combinations of the elements of $\mathcal{D}(\widetilde{A}^*)//\mathcal{D}(A)$ such that
$\sigma(0)=\tau'(0)=1$ and $\sigma'(0)=\tau(0)=0$. For $\rho\in\C$, define the function $\zeta_\rho(x):=\sigma(x)+\rho\tau(x)$ and let $\zeta_\infty(x):=\tau(x)$. In order to be able to use Corollary \ref{coro:wolfratshausen}, we will only consider $\mathcal{L}\zeta_\rho\in\ran(V_F)$, i.e. we can write $\mathcal{L}\zeta_\rho=V_F\phi$ for some $\phi\in\mathcal{D}(V_F)=\{f\in H^2(0,\infty), f(0)=0\}$. Let us therefore use the parameter $\rho\in\C\cup\{\infty\}$ and the function $\phi\in\mathcal{D}(V_F)$ to describe all extensions $A_{\rho,\phi}$ of the form
\begin{align*}
A_{\rho,\phi}:\qquad\mathcal{D}(A_{\rho,\phi})&=\mathcal{D}(A)\dot{+}\spann\{\zeta_\rho\}\\
\left(A_{\rho,\phi}(f+\lambda\zeta_\rho)\right)(x)&= -i(f''(x)+\lambda\zeta_\rho''(x))+W(x)(f(x)+\lambda\zeta_\rho(x))-\lambda\phi''(x)\:,
\end{align*}
where $f\in\mathcal{D}(A)$ and $\lambda\in\C$. Next, let us use Corollary \ref{coro:wolfratshausen} to find the conditions on $\rho$ and $\phi$ for $A_{\rho,\phi}$ to be dissipative. Firstly, observe that $V_K$ is the Neumann-Laplacian on the half-line. This can be seen from 
\begin{equation} \label{eq:king}
\langle f,V^*f\rangle=\overline{f(0)}f'(0)+\int_0^\infty|f'(x)|^2\text{d}x
\end{equation}
for all $f\in\mathcal{D}(V^*)$. In order to find the selfadjoint restrictions of $V^*$, observe that any additional selfadjoint boundary condition has to be of the form $f'(0)=rf(0)$, where $r\in\R$. The additional choice $r=\infty$ corresponds to a Dirichlet condition at $0$, i.e. $f(0)=0$ and describes the Friedrichs extension of $V$. For any $r<0$, we get that $\langle f,V^*f\rangle$ can be made negative, which therefore does not describe a non-negative selfadjoint extension of $V$. For $r\geq 0$, it is obvious that $r=0$ describes the smallest non-negative extension of $V$. Hence, the Kre\u\i n--von Neumann extension is given by the Neumann-Laplacian with domain
$\mathcal{D}(V_K)=\{f\in H^2(0,\infty), f'(0)=0\}$. It is also not hard to see that if we close $\mathcal{D}(V_K)$ with respect to the norm induced by \eqref{eq:king}, we get $\mathcal{D}(V_K^{1/2})=H^1(0,\infty)$. Now, since $\mathcal{D}(\widetilde{A}^*)//\mathcal{D}(A)\subset H^1(0,\infty)=\mathcal{D}(V_K^{1/2})$, we get that the first necessary condition from Corollary \ref{coro:wolfratshausen}, which requires that $\spann\{\zeta_\rho\}\subset\mathcal{D}(V_K^{1/2})$ in order for $A_{\rho,\phi}$ to be dissipative is satisfied for any $\rho\in\C\cup\{\infty\}$. Next, let us determine for which $\rho\in\C\cup\{\infty\}$ and $\phi\in\mathcal{D}(V_F)$ Condition \eqref{eq:scan3} is satisfied. For $\rho\in\C$, it reads as
\begin{align*}
&\Imag\langle \zeta_\rho,\widetilde{A}^*\zeta_\rho\rangle+\Imag\langle \zeta_\rho,V_F\phi\rangle\geq\frac{1}{4}\|V_K^{1/2}(\phi+2i\zeta_\rho)\|^2\\
\Leftrightarrow\:\:&\Imag\langle\zeta_\rho,-i\zeta_\rho''\rangle+\Imag\langle \zeta_\rho,-\phi''\rangle\geq\frac{1}{4}\|\phi'+2i\zeta_\rho'\|^2=\frac{1}{4}\|\phi'\|^2+\|\zeta_\rho'\|^2+\Real\langle \phi',i\zeta_\rho'\rangle\\
\Leftrightarrow\:\: & \Imag(\overline{\zeta_\rho(0)}i\zeta_\rho'(0))+\|\zeta_\rho'\|^2+\Imag\langle\phi'',\zeta_\rho\rangle\geq\frac{1}{4}\|\phi'\|^2+\|\zeta_\rho'\|^2+\Imag\langle\phi'',\zeta_\rho\rangle+\Imag(\overline{\phi'(0)}\zeta_\rho(0))\\
\Leftrightarrow\:\: & \Real \rho\geq \frac{1}{4}\|\phi'\|^2-\Imag(\phi'(0))\:.
\end{align*}
For $\rho=\infty$, we get the condition that
\begin{equation*}
0\geq\frac{1}{4}\|\phi'\|^2\:,
\end{equation*}
which means that the only allowed choice is $\phi(x)\equiv 0$ in this case.
\end{example}
\subsection{The strictly positive case} \label{subsec:strictly} Next, let us consider the case when the imaginary part $V$ is strictly positive, i.e.\ when there exists a positive number $\varepsilon>0$ such that
$\langle f,Vf\rangle\geq \varepsilon\|f\|^2$ for all $f\in\mathcal{D}(V)$. We introduce the notation $V\geq\varepsilon>0$ in this case. 
\begin{corollary} Let $(A,\widetilde{A})$ be a dual pair satisfying the common core property, where $A$ is dissipative. Moreover, let the imaginary part $V$ be strictly positive. Then, $A_{\mathcal{V,L}}$ is dissipative if and only if $\mathcal{V}\subset\mathcal{D}(V_K^{1/2})$ and for all $v\in\mathcal{V}$ we have that
\begin{equation} \label{eq:scan}
\Imag\langle v,\widetilde{A}^*v\rangle+\Imag\langle \mathcal{P}v,\mathcal{L}v\rangle\geq\frac{1}{4}\|V_F^{-1/2}\mathcal{L}v\|^2+\|V_K^{1/2}v\|^2\:.
\end{equation}
Here, $\mathcal{P}$ denotes the unbounded projection onto $\ker V^*$ along $\mathcal{D}(V_F^{1/2})$, according to the decomposition $\mathcal{D}(V_K^{1/2})=\mathcal{D}(V_F^{1/2})\dot{+}\ker V^*$.
\end{corollary}
\begin{proof} Since $V\geq\varepsilon>0$, we have that $\ran(V_F)=\ran(V_F^{1/2})=\mathcal{H}$, which means that the condition $\ran(\mathcal{L})\subset\ran(V_F^{1/2})$ is always satisfied. Thus, by Theorem \ref{thm:kuckuckswalzer}, it is necessary that $\mathcal{V}\subset\mathcal{D}(V_K^{1/2})$ for $A_{\mathcal{V,L}}$ to be dissipative.\\
Since $V\geq\varepsilon>0$, we have that $\mathcal{D}(V_K)=\mathcal{D}(\overline{V})\dot{+}\ker V^*$ with $V_K=V^*\upharpoonright_{\mathcal{D}(V_K)}$. This implies that $\ker V^*=\ker V_K$ and since $V_K$ is non-negative, we also get that $\ker V_K^{1/2}=\ker V^*$. It is known that $\mathcal{D}(V_K^{1/2})=\mathcal{D}(V_F^{1/2})\dot{+}\ker V^*$ (cf. \cite{Alonso-Simon}). Thus, we can rewrite 
\begin{align*}
&\frac{1}{4}\|\mathcal{U}V_F^{-1/2}\mathcal{L}v+2iV_K^{1/2}v\|^2=\frac{1}{4}\|\mathcal{U}V_F^{-1/2}\mathcal{L}v+2iV_K^{1/2}(\idty-\mathcal{P})v\|^2\\\overset{\eqref{eq:kuckuck}}{=}&\frac{1}{4}\|\mathcal{U}(V_F^{-1/2}\mathcal{L}v+2iV_F^{1/2}(\idty-\mathcal{P}))v\|^2=\frac{1}{4}\|V_F^{-1/2}\mathcal{L}v+2iV_F^{1/2}(\idty-\mathcal{P})v\|^2\\
=&\frac{1}{4}\|V_F^{-1/2}\mathcal{L}v\|^2+\|V_F^{1/2}(\idty-\mathcal{P})v\|^2+\Imag\langle V_F^{1/2}(\idty-\mathcal{P})v,V_F^{-1/2}\mathcal{L}v\rangle\\
=&\frac{1}{4}\|V_F^{-1/2}\mathcal{L}v\|^2+\|V_K^{1/2}v\|^2+\Imag\langle(\idty-\mathcal{P})v,\mathcal{L}v\rangle\:.
\end{align*}
With this, Condition \eqref{eq:gstinkert} from Theorem \ref{thm:alter} can be rewritten as
\begin{equation*}
\Imag\langle v,\widetilde{A}^*v\rangle+\Imag\langle \mathcal{P}v,\mathcal{L}v\rangle\geq \frac{1}{4}\|V_F^{-1/2}\mathcal{L}v\|^2+\|V_K^{1/2}v\|^2\:,
\end{equation*}
which is the desired result.
\end{proof}
\begin{remark} \normalfont From $V\geq\varepsilon$ it follows that both $V_F$ and $V_F^{1/2}$ are boundedly invertible and thus $\ran(V_F)=\ran(V_F^{1/2})=\mathcal{H}$. Hence, the strictly positive case is a special case of Section \ref{subsec:additionalreq}, since $\ran(\mathcal{L})\subset\ran(V_F)=\mathcal{H}$ is always satisfied. As in Corollary \ref{coro:wolfratshausen}, it is thus helpful to write $\mathcal{L}v=V_F\phi_v$ for any $v\in\mathcal{V}$, where $\phi_v$ is uniquely determined by $\mathcal{L}v$ since $V_F\geq\varepsilon$. Then, we can rewrite \eqref{eq:scan} as follows
\begin{equation} \label{eq:scanbesser}
\Imag\langle v,\widetilde{A}^*v\rangle+\Imag\langle \mathcal{P}v,V_F\phi_v\rangle\geq \frac{1}{4}\|V_F^{1/2}\phi_v\|^2+\|V_K^{1/2}v\|^2\:,
\end{equation}
which is more accessible to explicit computations. 
\end{remark}
\begin{example} \label{ex:1493} \normalfont Let $\mathcal{H}=L^2(0,1)$, assume that $\gamma\geq\sqrt{3}$ and consider the dual pair $(A_0,\widetilde{A}_0)$, given by
\begin{align*}
A_0:\qquad\mathcal{D}(A_0)=\mathcal{C}_c^\infty(0,1),\quad\left(A_0f\right)(x)=-if''(x)-\gamma\frac{f(x)}{x^2},\\
\widetilde{A}_0:\qquad\mathcal{D}(\widetilde{A}_0)=\mathcal{C}_c^\infty(0,1),\quad\left(\widetilde{A}_0f\right)(x)=if''(x)-\gamma\frac{f(x)}{x^2}\:.
\end{align*}
Define the dual pair $(A,\widetilde{A})$, where $A:=\overline{A_0}$ and $\widetilde{A}:=\overline{\widetilde{A}_0}$. By construction, $(A,\widetilde{A})$ has the common core property, where we choose $\mathcal{C}_c^\infty(0,1)=:\mathcal{D}$ to be the common core. The ``imaginary part" $V$ is given by 
\begin{align*}
V:\qquad\mathcal{D}(V)&=\mathcal{C}_c^\infty(0,1)\\
f&\mapsto-f''\:,
\end{align*}
which is a strictly positive operator, since its closure is a restriction of the Dirichlet-Laplacian on the unit interval: $$\langle f,Vf\rangle\geq \pi^2\|f\|^2\:\:\text{for all}\:\: f\in\mathcal{C}_c^\infty(0,1)\:.$$ Moreover, its adjoint $V^*$ is given by
\begin{equation*}
V^*:\qquad\mathcal{D}(V^*)=H^2(0,1)\:,\quad f\mapsto-f''
\end{equation*}
and its kernel is $\ker(V^*)=\spann\{1,x\}$. Thus, observe that for any $f\in H^2(0,1)$, the projection $\mathcal{P}$ onto $\ker(V^*)$ along $\mathcal{D}(V_F^{1/2})$ is given by
\begin{equation} \label{eq:unbddproj}
(\mathcal{P}f)(x)=(1-x)f(0)+xf(1)\:.
\end{equation}

 The choice $\gamma\geq\sqrt{3}$ ensures that $\dim \ker \widetilde{A}^*=\dim \ker A^*=1$, which keeps the extension problem simpler. It can be shown by straightforward calculation that $\mathcal{D}(\widetilde{A}^*)$ can be written as
\begin{equation*}
\mathcal{D}(\widetilde{A}^*)=\mathcal{D}(A)\dot{+}\spann\{x^{\omega},x^{\overline{\omega}+2}\}\:,
\end{equation*}
where we have defined $\omega:=(1+\sqrt{1+4i\gamma})/2$. We therefore choose $\mathcal{D}(\widetilde{A}^*)//\mathcal{D}(A)=\spann\{x^{\omega},x^{\overline{\omega}+2}\}$. Let us now parametrize all proper ``one-dimensional" extensions of $(A,\widetilde{A})$, with the family of operators $\{A_\rho\}_{\rho\in\C\cup\{\infty\}}$ given by
\begin{equation*}
A_\rho:\qquad\mathcal{D}(A_\rho)=\mathcal{D}(A)\dot{+}\spann\{\xi_\rho\},\quad A_\rho=\widetilde{A}^*\upharpoonright_{\mathcal{D}(A_\rho)}\:,
\end{equation*}
where $$\spann\{x^{\omega},x^{\overline{\omega}+2}\}\ni\xi_\rho(x):=\begin{cases}\rho\left(\frac{(2+\overline{\omega})x^{\omega}-\omega x^{\overline{\omega}+2}}{2+\overline{\omega}-\omega}\right)-\frac{x^{\omega}-x^{\overline{\omega}+2}}{2+\overline{\omega}-\omega}\quad&\text{for}\:\:\rho\in\C\\\frac{(2+\overline{\omega})x^{\omega}-\omega x^{\overline{\omega}+2}}{2+\overline{\omega}-\omega}&\text{for}\:\:\rho=\infty\end{cases}$$ satisfies the boundary conditions
\begin{align*}
\xi_\rho(0)=\xi_\rho'(0)=0\:\:\:&\text{for}\:\:\:\rho\in\C\cup\{\infty\}\\
\xi_\rho(1)=\rho,\:\: \xi_\rho'(1)=1\:\:\:&\text{for}\:\:\:\rho\in\C \:\:\:\text{and}\:\:\:\xi_\rho(1)=1,\:\: \xi'_\rho(1)=0\:\:\:\text{for}\:\:\:\rho=\infty\:.
\end{align*}
Next, \eqref{eq:unbddproj} implies that for $\rho\in\C$, we get $\mathcal{P}\xi_\rho (x)=\rho x$, whereas for $\rho=\infty$, we get $\mathcal{P}\xi_\infty(x)=x$. This follows from the fact that $\mathcal{D}(V_F^{1/2})=H^1_0(0,1)$ and for any $\rho\in\C$, we have $\xi_\rho(0)=\xi_\infty(0)=0$ as well as $\xi_\rho(1)=\rho$ and $\xi_\infty(1)=1$. Now, since $V$ is strictly positive, we know that its Friedrichs extension $V_F$ is bijective, which means that any function $\mathcal{L}\xi_\rho\in L^2(0,1)$ can be written as $\mathcal{L}\xi_\rho=V_F\phi=-\phi''$ for some unique $\phi\in\mathcal{D}(V_F)=\{\phi\in H^2(0,1), \phi(0)=\phi(1)=0\}$.
Hence, let us use the parameter $\rho\in\C\cup\{\infty\}$ and the arbitrary function $\phi\in\mathcal{D}(V_F)$ to label all one-dimensional extensions of $\mathcal{D}(A)$ that have domain contained in $\mathcal{D}(\widetilde{A}^*)$. They are given by
\begin{align*}
A_{\rho,\phi}:\qquad\qquad\mathcal{D}(A_{\rho,\phi})&=\mathcal{D}(A)\dot{+}\spann\{\xi_\rho\}\\
[A_{\rho,\phi}(f+\lambda\xi_\rho)](x)&=(-if''(x)-\lambda i\xi_\rho''(x))-\gamma\frac{f(x)+\lambda\xi_\rho(x)}{x^2}-\lambda \phi''(x)\:,
\end{align*}
where $f\in\mathcal{D}(A)$ and $\lambda\in\C$. By \eqref{eq:scanbesser}, we have that $A_{\rho,\phi}$ is dissipative if and only if
\begin{align*}
\Imag\langle \xi_\rho,\widetilde{A}^*\xi_\rho\rangle-\|V_K^{1/2}\xi_\rho\|^2\geq\frac{1}{4}\|V_F^{1/2}\phi\|^2-\Imag\langle \mathcal{P}\xi_\rho,V_F\phi\rangle\:.
\end{align*}
is satisfied. Using that for any $v\in \mathcal{D}(V_K^{1/2})=H^1(0,1)$, we have
\begin{equation*}
\|V_K^{1/2}v\|^2=\|v'\|^2-|v(1)-v(0)|^2\:,
\end{equation*}
it can be shown that for any $v\in\spann\{x^{\omega},x^{\overline{\omega}+2}\}$, we have $$\Imag\langle v,\widetilde{A}^*v\rangle-\|V_K^{1/2}v\|^2=-\Real\left(\overline{v(1)}v'(1)\right)+|v(1)|^2\:,$$ 
which means that
\begin{equation*}
\Imag\langle\xi_\rho,\widetilde{A}^*\xi_\rho\rangle-\|V_K^{1/2}\xi_\rho\|^2=\begin{cases} |\rho|^2-\Real(\rho)\quad&\text{if}\quad \rho\in\C\\ 1\quad&\text{if}\quad \rho=\infty\:.\end{cases}
\end{equation*}
Moreover, since $\|V_F^{1/2}\phi\|=\|\phi'\|$ and 
\begin{equation}
\Imag\left(\int_0^1 x\phi''(x)\text{d}x\right)=\Imag(\phi'(1))
\end{equation}
for any $\phi\in\mathcal{D}(V_F)$, the above yields the conditions on $\rho$ and $\phi$ for $A_{\rho,\phi}$ to be dissipative:
\begin{align*}
&\frac{1}{4}\|\phi'\|^2+\Imag(\overline{\rho}\phi'(1))\leq |\rho|^2-\Real\rho\quad\text{for}\quad\rho\in\C\\
&\frac{1}{4}\|\phi'\|^2+\Imag(\phi'(1))\leq 1\qquad\qquad\quad\:\:\text{for}\quad\rho=\infty\:.
\end{align*}
For the case of proper extensions, where $\phi=0$, i.e.\ for $A_{\rho,0}$ we therefore have the condition that either $\rho=\infty$ or $|\rho|^2-\Real\rho\geq 0$ for $A_{\rho,0}$ to be dissipative. In the non-proper case, for a suitable choice of $\phi$, it is no longer necessary that $\rho$ satisfies this condition. For instance, let $\phi(x):=(x^2-x)\in\mathcal{D}(V_F)$. We then get the condition
\begin{equation*}
\frac{1}{4}\|\phi'\|^2+\Imag(\overline{\rho}\phi'(1))=\frac{1}{12}-\Imag(\rho)\leq |\rho|^2-\Real\rho
\end{equation*}
for $A_{\rho,(x^2-x)}$ to be dissipative. This condition is for example satisfied by $\rho=\frac{1}{2}+\frac{3}{8}i$, i.e. $A_{\left(\frac{1}{2}+\frac{3}{8}i\right),(x^2-x)}$ is dissipative, while $A_{\left(\frac{1}{2}+\frac{3}{8}i\right),0}$ is not. In Corollary \ref{coro:ausmachen}, we will show that the phenomenon that we have a dissipative non-proper extension, defined on a domain on which the corresponding proper extension would not be dissipative, can only occur if the Friedrichs and Kre{\u\i}n-von Neumann extensions of $V$ do not coincide, as it is the case in this example.
\label{ex:shirley}
\end{example}
\begin{remark} \normalfont The choice of the highly singular $x^{-2}$-potential allowed us to compute everything explicitly. It is however not very difficult to add a ``small" extra potential.
\end{remark}
\subsection{The case of coinciding Friedrichs and Kre{\u\i}n-von Neumann extension} \label{subsec:essentials}
Let us now consider the case that the Friedrichs and Kre{\u\i}n-von Neumann extensions of $V$ coincide: $V_F=V_K=:\widehat{V}$. Before we simplify Theorem \ref{thm:alter} with the help of this extra assumption, let us prove that in this case, both conditions, $\mathcal{V}\subset\mathcal{D}(\widehat{V}^{1/2})$ and $\ran(\mathcal{L})\subset\ran( \widehat{V}^{1/2})$ are \emph{independently} necessary for $A_{\mathcal{V,L}}$ to be dissipative.
\begin{lemma} \label{lemma:singulariter} Let $(A,\widetilde{A})$ be a dual pair satisfying the common core condition, where $A$ is dissipative and assume in addition that for the imaginary part $V$ we have $V_F=V_K=:\widehat{V}$. Then, for $A_{\mathcal{V,L}}$ to be dissipative it is necessary that $\mathcal{V}\subset\mathcal{D}(\widehat{V}^{1/2})$ and $\ran(\mathcal{L})\subset\ran(\widehat{V}^{1/2})$.
\end{lemma}
\begin{proof} We only need to show that $\mathcal{V}\subset\mathcal{D}(\widehat{V}^{1/2})$ is necessary for $A_{\mathcal{V,L}}$ to be dissipative. The condition $\ran(\mathcal{L})\subset\ran(\widehat{V}^{1/2})$ will then just follow from Theorem \ref{thm:kuckuckswalzer}, ii).
Thus, assume that there exists a $v\in\mathcal{V}$ such that $v\notin\mathcal{D}(\widehat{V}^{1/2})$. 
Since for any $f\in\mathcal{D}(A)$, $v\in\mathcal{V}$ we have
\begin{equation*}
\Imag\langle f+v,A_\mathcal{V,L}(f+v)\rangle=\langle f,Vf\rangle+\Imag\langle v,2iVf\rangle+\Imag\langle v,(\widetilde{A}^*+\mathcal{L})v\rangle+\Imag\langle f,\mathcal{L}v\rangle\:,
\end{equation*}
showing that
\begin{equation} \label{eq:irish} \inf_{f\in\mathcal{D}({V})}(\langle f,{V}f\rangle+\Imag(\langle v,2i{V}f\rangle-\langle \mathcal{L}v,f\rangle))=-\infty\:,
\end{equation}
will imply that $A_{\mathcal{V,L}}$ cannot be dissipative. We will proceed to show that 
\begin{equation} \label{eq:irish2} \inf_{f\in\mathcal{D}(\overline{V})}(\langle f,\overline{V}f\rangle+\Imag(\langle v,2i\overline{V}f\rangle-\langle \mathcal{L}v,f\rangle))=-\infty
\end{equation}
and using that $\mathcal{D}(V)$ is a core for $\overline{V}$, this implies that for each $\widetilde{f}_n\in\mathcal{D}(\overline{V})$ we can choose a sequence $\{\widetilde{f}_{n,m}\}_{m=1}^\infty\subset\mathcal{D}(V)$ such that $\widetilde{f}_{n,m}\overset{m\rightarrow\infty}{\longrightarrow}\widetilde{f}_n$ and $V\widetilde{f}_{n,m}\overset{m\rightarrow\infty}{\longrightarrow}\overline{V}\widetilde{f}_n$. A diagonal sequence argument then shows \eqref{eq:irish} and thus the lemma.

Let us thus now show \ref{eq:irish2}. To this end, let $P$ denote the projection-valued spectral measure corresponding to $\widehat{V}$ and define $P_1:=P([0,1))$ and $P_2:=P([1,\infty))$ as well as $\mathcal{H}_{1,2}:=P_{1,2}\mathcal{H}$. Since $\widehat{V}\geq 0$, we have $P_1+P_2=\idty$, resp. $\mathcal{H}_1\oplus\mathcal{H}_2=\mathcal{H}$.
Now, observe that $v\notin\mathcal{D}(\widehat{V}^{1/2})$ if and only if $P_2v\notin\mathcal{D}(\widehat{V}^{1/2})$.
Since $V_F^{1/2}=V_K^{1/2}=\widehat{V}^{1/2}$, we get that $\ran (V_F^{1/2}\upharpoonright_{\mathcal{D}(V)})=\ran(\widehat{V}^{1/2}\upharpoonright_{\mathcal{D}(V)})$. Hence, if $P_2v\notin\mathcal{D}(\widehat{V}^{1/2})=\mathcal{D}(V_K^{1/2})$, we have by Corollary \ref{coro:andonishio} that there exists a sequence $\{f_n\}\subset\mathcal{D}(\overline{V})$ such that $\|\widehat{V}^{1/2}f_n\|=1$ for all $n\in\N$ and 
\begin{equation} \label{eq:kuzel}
\lim_{n\rightarrow\infty}|\langle P_2v,\widehat{V}^{1/2}\widehat{V}^{1/2}f_n\rangle|=+\infty\:.
\end{equation}
We now claim that the sequence $\{\widehat{V}^{1/2}P_2{f}_n\}$ satisfies
$$ \|\widehat{V}^{1/2}P_2{f}_n\|\leq 1\quad\text{and}\quad\lim_{n\rightarrow\infty}|\langle v,\widehat{V}^{1/2}\widehat{V}^{1/2}P_2{f}_n\rangle|=+\infty\:.$$
The first statement follows immediately from 
\begin{equation} \label{eq:estimate2}
\|\widehat{V}^{1/2}P_2{f}_n\|=\|P_2\widehat{V}^{1/2}{f}_n\|\leq\|\widehat{V}^{1/2}f_n\|=1\:,
\end{equation}
while the second statement follows from \eqref{eq:kuzel}.
Next, observe that
\begin{equation} \label{eq:biscuits}
\|P_2{f}_n\|^2=\int_{[1,\infty)}\text{d}\|P(\lambda){f}_n\|^2\leq\int_{[1,\infty)}\lambda\text{d}\|P(\lambda){f}_n\|^2=\|\widehat{V}^{1/2}P_2{f}_n\|^2\overset{\eqref{eq:estimate2}}{\leq}1\:.
\end{equation}
For any $n\in\N$ choose $\phi_n\in[0,2\pi)$ such that $$\Imag\langle v,\widehat{V}^{1/2}\widehat{V}^{1/2}e^{i\phi_n}P_2{f}_n\rangle=-|\langle v,\widehat{V}^{1/2}\widehat{V}^{1/2}P_2{f}_n\rangle|\:.$$ Choosing $g_n=e^{i\varphi_n}P_2f_n$ for any $n\in\N$ now yields \eqref{eq:irish2} since
\begin{align*}
&\langle g_n,Vg_n\rangle+\Imag(\langle v,2iVg_n\rangle-\langle\mathcal{L}v,g_n\rangle)=\|\widehat{V}^{1/2}P_2f_n\|^2-2|\langle v, \widehat{V}^{1/2}\widehat{V}^{1/2}P_2f_n\rangle|-\Imag\langle \mathcal{L}v,e^{i\varphi_n}P_2f_n\rangle\\
\overset{\eqref{eq:estimate2}}{\leq} & 1-2|\langle v, \widehat{V}^{1/2}\widehat{V}^{1/2}P_2f_n\rangle|+\|\mathcal{L}v\|\|P_2f_n\|\overset{\eqref{eq:biscuits}}{\leq} 1-2|\langle v, \widehat{V}^{1/2}\widehat{V}^{1/2}P_2f_n\rangle|+\|\mathcal{L}v\|\overset{n\rightarrow\infty}{\longrightarrow}-\infty\:.
\end{align*}
This finishes the proof.
\end{proof}
\begin{remark} \normalfont This result applies in particular to the case of $V$ being essentially selfadjoint, where we have $\overline{V}=V_F=V_K$. However, note that the previous lemma and the following corollary cover a wider class of imaginary parts $V$ than just the essentially selfadjoint ones. For example, let $\mathcal{H}=L^2(\R^+)$ and consider the imaginary part $V$ given by
\begin{align*}
V:\mathcal{D}(V)=\mathcal{C}_c^\infty(\R^+),\qquad\left(Vf\right)(x)=-f''(x)-\frac{1/4}{x^2}f(x)\:.
\end{align*}
It is a well-known fact that $V$ is not essentially selfadjoint but that its Friedrichs and Kre{\u\i}n-von Neumann extension coincide (cf.\ e.g.\ \cite[Prop.\ 4.21]{BDG11}). For an abstract criterion as to whether a non-negative symmetric and non-essentially selfadjoint operator has a unique non-negative selfadjoint extension, we refer to Kre{\u\i}n's result in \cite{Krein} and its presentation in \cite[Thm.\ 2.12]{Alonso-Simon}. \label{rem:referee}
\end{remark}
Let us now simplify Theorem \ref{thm:alter} for the case that the Friedrichs and Kre{\u\i}n-von Neumann extension $V_F$ and $V_K$ of $V$ coincide: $V_F=V_K=:\widehat{V}$. Note that the situation of $V$ being essentially selfadjoint is a special case of this. We also want to show that $A_{\mathcal{V,L}}$ can only be dissipative if $A_{\mathcal{V}}$ already is, i.e. there necessarily needs to be a dissipative boundary condition -- described by a suitable choice of $\mathcal{V}$ -- before one can consider deviations from the action of $\widetilde{A}^*$ via non-zero operators $\mathcal{L}$. This is fundamentally different to the case $V_F\neq V_K$, where we have found an example of an extension $A_{\mathcal{V,L}}$, which was dissipative while $A_\mathcal{V}$ was not (Example \ref{ex:1493}).
\begin{corollary} \label{coro:ausmachen} Let $(A,\widetilde{A})$ be dual pair satisfying the common core property, where $A$ is dissipative. Moreover, let the imaginary part $V$ be such that its Friedrichs and Kre{\u\i}n-von Neumann extension coincide, i.e.\ $V_F=V_K=:\widehat{V}$. Then, $A_{\mathcal{V,L}}$ is dissipative if and only if $\mathcal{V}\subset\mathcal{D}(\widehat{V}^{1/2})$, $\ran(\mathcal{L})\subset\ran(\widehat{V}^{1/2})$ and for all $v\in\mathcal{V}$ we have that
\begin{equation} \label{eq:scan2}
\Imag\langle v,\widetilde{A}^*v\rangle\geq\frac{1}{4}\|\widehat{V}^{\,-1/2}\mathcal{L}v\|^2+\|\widehat{V}^{1/2}v\|^2\:.
\end{equation}
In particular, this implies that for $A_{\mathcal{V,L}}$ to be dissipative, it is necessary that $A_\mathcal{V}$ is dissipative. 
\end{corollary}
\begin{proof} The conditions that $\mathcal{V}\subset\mathcal{D}(\widehat{V}^{1/2})$ and $\ran(\mathcal{L})\subset\ran(\widehat{V}^{1/2})$ for $A_{\mathcal{V,L}}$ to be dissipative follow from Lemma \ref{lemma:singulariter}. Condition \eqref{eq:scan2} follows from \eqref{eq:gstinkert} using that $V_K=V_F=\widehat{V}$, which implies that $\mathcal{U}$ acts like the identity on $\overline{\ran(\widehat{V}^{1/2})}$. Moreover, for $\mathcal{L}=0$, we get that $A_\mathcal{V}$ is dissipative if and only if $\mathcal{V}\subset\mathcal{D}(\widehat{V}^{1/2})$ and $\Imag\langle v,\widetilde{A}^*v\rangle\geq\|\widehat{V}^{1/2}v\|^2$ for all $v\in\mathcal{V}$. Thus, if $A_\mathcal{V}$ is not dissipative then it is either true that $\mathcal{V}\not\subset\mathcal{D}(\widehat{V}^{1/2})$ or we have $\mathcal{V}\subset\mathcal{D}(\widehat{V}^{1/2})$ but there exists a $v\in\mathcal{V}$ such that
$$ \Imag\langle v,\widetilde{A}^*v\rangle-\|\widehat{V}^{1/2}v\|^2<0\:,$$
both implying that $A_{\mathcal{V,L}}$ cannot be dissipative either. 
This shows the corollary.
\end{proof}
\begin{example} \label{ex:konzert} \normalfont Let $0<\gamma<1/2$ and consider the dual pair of operators
\begin{align*}
A_0:\quad\mathcal{D}(A_0)&=\mathcal{C}_c^\infty(0,1),\quad (A_0f)(x):=if'(x)+\frac{i\gamma}{x}f(x)\\
\widetilde{A}_0:\quad\mathcal{D}(\widetilde{A}_0)&=\mathcal{C}_c^\infty(0,1),\quad (\widetilde{A}_0f)(x):=if'(x)-\frac{i\gamma}{x}f(x)\:,
\end{align*}
where $A_0$ is dissipative and $\widetilde{A}_0$ is antidissipative. We denote their closures by $A=\overline{A_0}$ and $\widetilde{A}=\overline{\widetilde{A}_0}$. It can be shown that $$\mathcal{D}(\widetilde{A}^*)=\mathcal{D}(A)\dot{+}\spann\{x^{-\gamma},x^{\gamma+1}\}$$ and we therefore choose $\mathcal{D}(\widetilde{A}^*)//\mathcal{D}(A)=\spann\{x^{-\gamma},x^{\gamma+1}\}$. Moreover, it is easy to see that the imaginary part $V$ is the essentially selfadjoint multiplication operator by the function $\frac{\gamma}{x}$ with domain $\mathcal{C}_c^\infty(0,1)$ which has closure to the selfadjoint maximal multiplication operator by $\frac{\gamma}{x}$, which we denote by $\overline{V}$. Since $x^{-\gamma}\notin\mathcal{D}(\overline{V}^{1/2})$, this means that the only choice for $\mathcal{V}\subset\spann\{x^{-\gamma},x^{\gamma+1}\}$ in order to have a chance for $A_{\mathcal{V,L}}$ to be dissipative is $\mathcal{V}:=\spann\{x^{\gamma+1}\}$. Let us define $v(x)=:x^{\gamma+1}$ and $\mathcal{L}v=:\ell\in\mathcal{H}$ and let us use the functions $v$ and $\ell$ to label $A_{\mathcal{V,L}}=:A_{v,\ell}$. Since $\langle f,\overline{V}f\rangle\geq\gamma\|f\|^2$ for all $f\in\mathcal{D}(\overline{V})$, we get that $\overline{V}$ and $\overline{V}^{1/2}$ are both boundedly invertible, in particular that $\ran(\overline{V}^{1/2})=\mathcal{H}$. Thus, by Corollary \ref{coro:ausmachen}, it only remains to check whether Condition \eqref{eq:scan2} is satisfied, which reads as
\begin{equation*}
\Imag\langle v,\widetilde{A}^*v\rangle-\|\overline{V}^{1/2}v\|^2\geq\frac{1}{4}\|\overline{V}^{-1/2}\ell\|^2\:.
\end{equation*}
It can be easily shown that 
\begin{equation*}
\Imag\langle v,\widetilde{A}^*v\rangle-\|\overline{V}^{1/2}v\|^2{=}\frac{1}{2}\left(|v(1)|^2-|v(0)|^2\right)=\frac{1}{2}\:.
\end{equation*}
Hence, $A_{v,\ell}$ is dissipative if and only if
\begin{equation*}
\|\overline{V}^{\:-1/2}\ell\|^2=\frac{1}{\gamma}\int_0^1 x|\ell(x)|^2\text{d}x\leq 2\:.
\end{equation*}
This means that all dissipative extensions of $A$ that have domain contained in $\mathcal{D}(\widetilde{A}^*)$ are given by
\begin{align} \label{eq:referee}
A_{v,\ell}:\qquad\qquad\qquad\mathcal{D}(A_{v,\ell})&=\mathcal{D}(A)\dot{+}\spann\{v\}\notag\\
\left(A_{v,\ell}(f+\lambda v)\right)(x)&= if'(x)+i\lambda v'(x)+i\gamma\frac{f(x)+\lambda v(x)}{x}+\lambda \ell(x)\:,
\end{align}
where $f\in\mathcal{D}(A)$ and $\lambda\in\C$. The function $\ell\in L^2(0,1)$ has to satisfy
\begin{equation} \label{eq:erdkreis}
\int_0^1x|\ell(x)|^2\text{d}x\leq 2\gamma\:.
\end{equation}
Moreover, by Lemma \ref{prop:dampfnudel}, we have that $A_{v,\ell}$ is maximally dissipative since it is a one-dimensional extension of $A$. 
\begin{remark} As in Example \ref{ex:potsdam}, one could easily add a sufficiently ``small" real potential $W$ to the operators $A_0$ and $\widetilde{A}_0$. 
\end{remark}
\end{example}
\section{Operators with bounded imaginary part}
In this section, we will apply the result of Corollary \ref{coro:ausmachen} in order to construct all dissipative extensions of a dissipative operator with bounded imaginary part, where $\mathcal{D}(A)=\mathcal{D}(S)=\mathcal{D}(V)$. While this is not a new result (it can for example essentially be found in \cite[Theorem 1]{Crandall} with a different way of proof), we want to give more attention to the interplay between boundary conditions and bounded dissipative perturbations. In particular, we will show that if an operator with non-dissipative boundary condition is considered, it is impossible to add a bounded dissipative perturbation such that the result is a dissipative operator (Corollary \ref{coro:jsbach}, ii). On the other hand, we will show that the ``more dissipative" a boundary condition is, the more freedom one has in describing dissipative extensions of a given operator (Corollary \ref{coro:jsbach}, iii).

To start with, let us show that it is sufficient to only consider operators of the form $S+iV$, where $S$ is symmetric and $V\geq 0$ is bounded:
\begin{lemma} \label{lemma:uab} Let $A$ be a dissipative operator and assume that the quadratic form $q$ given by
\begin{equation*}
q:\quad\mathcal{D}(q)=\mathcal{D}(A),\quad f\mapsto\Imag\langle f,Af\rangle
\end{equation*}
is bounded. Then there exists a symmetric operator $S$ with $\mathcal{D}(S)=\mathcal{D}(A)$ and a essentially selfadjoint bounded operator $V\geq 0$ with $\mathcal{D}(V)=\mathcal{D}(A)$ such that $A=S+iV$.
\end{lemma}
\begin{proof}
Since $q$ is bounded, it is closable. Let $V'$ denote the bounded selfadjoint operator associated to it and let $V=V'\upharpoonright_{\mathcal{D}(A)}$. It is not hard to see that $S:=A-iV$ is symmetric and trivially $A=S+iV$.
\end{proof}
Next, let us show that for \emph{any} dissipative extension of $S+iV$, it is necessary that its domain is contained in $\mathcal{D}(S^*)$. Since $V$ is assumed to be bounded, note that $A=S+iV$ is closed if and only if $S$ is closed. Also note that we are describing general extensions of $A$ that need \emph{not} be of the form $\widehat{S}+iV$, where $\widehat{S}$ is a symmetric extension of $S$.
\begin{lemma} \label{lemma:allextensions}
Let $A:=S+iV$, where $S$ is closed and symmetric and $V\geq 0$ is bounded. Then, for an extension $A\subset B$ to be dissipative, it is necessary that $\mathcal{D}(B)\subset\mathcal{D}(S^*)$.
\end{lemma}
\begin{proof} Assume that $\mathcal{D}(B)\not\subset\mathcal{D}(S^*)$, i.e. that there exists a $v\in\mathcal{D}(B)$ such that $v\notin\mathcal{D}(S^*)$. For any $f\in\mathcal{D}(A)=\mathcal{D}(S)$, consider
\begin{align} \label{eq:animal}
\Imag\langle f+v,B(f+v)\rangle&=\Imag\langle f,(S+iV)f\rangle+\Imag\langle v,(S+iV)f\rangle+\Imag\langle f+v,Bv\rangle\notag\\
&=\langle f,Vf\rangle+\Imag\langle v,Sf\rangle+\Imag\langle v,iVf\rangle+\Imag\langle f+v,Bv\rangle\notag\\
&\leq \|V\|\|f\|^2+\Imag\langle v,Sf\rangle+\|V\|\|v\|\|f\|+\|f\|\|Bv\|+\|v\|\|Bv\|\:.
\end{align}
Since $v\notin\mathcal{D}(S^*)$, there exists a normalized sequence $\{f_n\}\subset\mathcal{D}(S)$ such that $$\lim_{n\rightarrow\infty}\Imag\langle v,Sf_n\rangle=-\infty\:.$$ Using \eqref{eq:animal}, we therefore get
\begin{align*}
\Imag\langle f_n+v,B(f_n+v)\rangle\leq \|V\|+\|V\|\|v\|+\|Bv\|+\|v\|\|Bv\|+\Imag\langle v,Sf_n\rangle\overset{n\rightarrow\infty}{\longrightarrow}-\infty\:,
\end{align*}
which shows that $B$ cannot be dissipative in this case. This finishes the proof.
\end{proof}
We are now able to describe all dissipative extensions of $A=S+iV$:
\begin{theorem} \label{thm:bern} Let $A=S+iV$ be a dissipative operator with bounded imaginary part. Then $S_\mathcal{V,L}+iV$, where $S_{\mathcal{V,L}}$ is defined as in Definition \ref{def:quantz}, is a dissipative extension of $S+iV$ if and only if for all $v\in\mathcal{V}\subset\mathcal{D}(S^*)//\mathcal{D}(S)$ we have that $\mathcal{L}v\in\ran(\overline{V}^{1/2})$ and the condition
\begin{equation} \label{eq:ragout}
\Imag\langle v,S^*v\rangle\geq\frac{1}{4}\|\overline{V}^{-1/2}\mathcal{L}v\|^2
\end{equation}
is satisfied.
As before, $\overline{V}^{-1/2}$ denotes the inverse of $\overline{V}^{1/2}$ on the reducing subspace $\overline{\text{\emph{ran}}(\overline{V}^{1/2})}$ as described in \eqref{eq:krakau}. Moreover, \emph{all} dissipative extensions of $S+iV$ are of this form.
\end{theorem}
\begin{remark} For the case that $V\equiv 0$, this means in particular that Condition \eqref{eq:ragout} simplifies to the condition that $\Imag\langle v,S^*v\rangle\geq 0$ for all $v\in\mathcal{V}$.
\end{remark}
\begin{proof} Since $V$ is bounded, $S_{\mathcal{V,L}}$ is an extension of $S$ if and only if $A_{\mathcal{V,L}}=S_{\mathcal{V,L}}+iV$ is an extension of $A:=S+iV$. Clearly, for $A:=S+iV$ and $\widetilde{A}:=S-iV$, we have that $(A,\widetilde{A})$ is a dual pair and we get that $\mathcal{D}(A)=\mathcal{D}(\widetilde{A})=\mathcal{D}(S)$, which means that it has the common core property. Moreover, by boundedness of $V$, we get that $\widetilde{A}^*=S^*+iV$, where $\mathcal{D}(\widetilde{A}^*)=\mathcal{D}(S^*)$. Also, observe that $V\upharpoonright_{\mathcal{D}(S)}$ is essentially selfadjoint, which means that we can apply Corollary \ref{coro:ausmachen}. Since $V$ is bounded, we have that $\mathcal{D}(\overline{V}^{1/2})=\mathcal{D}(\overline{V} )=\mathcal{H}$, which means that the Condition that $\mathcal{V}\subset\mathcal{D}(\overline{V}^{1/2})$ is always satisfied. Thus, by Corollary \ref{coro:ausmachen}, it is necessary that $\ran(\mathcal{L})\subset\ran(\overline{V}^{1/2})$ for $A_{\mathcal{V,L}}$ to be dissipative. Condition \eqref{eq:scan2} reads as
\begin{align*} \Imag\langle v,(S^*+iV)v\rangle\geq\|\overline{V}^{1/2}v\|^2+\frac{1}{4}\|\overline{V}^{-1/2}\mathcal{L}v\|^2
\Leftrightarrow\:\Imag\langle v,S^*v\rangle\geq\frac{1}{4}\|\overline{V}^{-1/2}\mathcal{L}v\|^2\:,
\end{align*} 
which is the desired result. Let us finish by showing that \emph{all} dissipative extensions of $S+iV$ are parametrized by the operators $S_{\mathcal{V,L}}+iV$. By Lemma \ref{lemma:allextensions}, we know that all dissipative extensions have domain contained in $\mathcal{D}(S^*)=\mathcal{D}(\widetilde{A}^*)$. On the other hand, since $\mathcal{V}$ is an arbitrary subspace of $\mathcal{D}(S^*)//\mathcal{D}(S)$, the extensions $S_{\mathcal{V,L}}$ describe all possible extensions of $S$ that have domain contained in $\mathcal{D}(S^*)$. As they are dissipative if and only if $\mathcal{V}$ and $\mathcal{L}$ satisfy the assumptions of this theorem, we have found all dissipative extensions of $(S+iV)$.
\end{proof}
Let us now investigate the relation between the choice of $\mathcal{V}$ and $\mathcal{L}$:
\begin{corollary} \label{coro:jsbach} Let $\mathcal{V}\subset\mathcal{D}(S^*)//\mathcal{D}(S)$.\\
i) If $S_\mathcal{V}$ is symmetric, then $(S_\mathcal{V}+iV)$ is the only dissipative extension of $(S+iV)$ with domain equal to $\mathcal{D}(S_\mathcal{V})$. Moreover, the imaginary part of any other extension of the form $(S_{\mathcal{V,L}}+iV)$ is not bounded from below, i.e. for $\mathcal{L}\neq 0$, there exists no $\gamma\in\R^+$ such that
\begin{equation} \label{eq:semibounded} \inf_{\psi\in\mathcal{D}(S_{\mathcal{V,L}}):\|\psi\|=1}\Imag\langle \psi,(S_{\mathcal{V,L}}+iV)\psi\rangle\geq -\gamma\|\psi\|^2\:.
\end{equation}\\
ii) If $S_\mathcal{V}$ is not dissipative, i.e. if there exists a $v\in\mathcal{V}$ such that
\begin{equation*}
\Imag\langle v,S_\mathcal{V}v\rangle<0\:,
\end{equation*}
then there exists no extension $S_{\mathcal{V,L}}$ and no bounded non-negative operator $V\geq 0$ such that $S_{\mathcal{V,L}}+iV$ is dissipative.\\
iii) If there exists an $\varepsilon>0$ such that
$$\Imag\langle v,S_\mathcal{V}v\rangle \geq \varepsilon \|v\|^2$$
for all $v\in\mathcal{V}$ and if the operator $\mathcal{L}$ is bounded, we get that
\begin{equation} \label{eq:vondenleyen}
\Imag\langle \psi, S_{\mathcal{V,L}}\psi\rangle\geq-\frac{\|\mathcal{L}\|^2}{4\varepsilon}\|\psi\|^2
\end{equation}
for all $\psi\in\mathcal{D}(S_{\mathcal{V,L}})$.
This implies in particular that for any bounded $V\geq \frac{\|\mathcal{L}\|^2}{4\varepsilon}$, we get
\begin{equation*}
\Imag\langle \psi, (S_{\mathcal{V,L}}+iV)\psi\rangle\geq 0
\end{equation*}
for all $\psi\in\mathcal{D}(S_{\mathcal{V,L}})$.
\end{corollary}
\begin{proof} i) By Theorem \ref{thm:bern}, Condition \eqref{eq:ragout}, it is necessary that
\begin{equation*}
\Imag\langle v,S^*v\rangle\geq\frac{1}{4}\|V^{-1/2}\mathcal{L}v\|^2
\end{equation*}
for all $v\in\mathcal{V}$. But since $S_\mathcal{V}=S^*\upharpoonright_{\mathcal{D}(S_\mathcal{V})}$ is symmetric, we get $\Imag\langle v,S^*v\rangle= 0$ for all $v\in\mathcal{V}$, which makes it necessary that $\mathcal{L}v=0$ for all $v\in\mathcal{V}$ for $(S_{\mathcal{V,L}}+iV)$ to be dissipative. In other words, only for $\mathcal{L}\equiv 0$ do we have that $A_{\mathcal{V,L}=0}=(S_{\mathcal{V,L}=0}+iV)$ is dissipative. For the second part of i), assume that the imaginary part of $A_{\mathcal{V,L}}$ is semibounded with semibound $-\gamma$ (cf.\ \eqref{eq:semibounded}). This would mean that the operator $S_{\mathcal{V,L}}+i(V+\gamma)$ is dissipative, which by Condition \eqref{eq:ragout} would imply that for all $v\in\mathcal{V}$, the condition
\begin{equation*}
0= \Imag\langle v,S^*v\rangle\geq\frac{1}{4}\|(V+\gamma)^{-1/2}\mathcal{L}v\|^2\:,
\end{equation*}
is satisfied, which is impossible if $\mathcal{L}\neq 0$.

ii) Let $v$ be an element of $\mathcal{V}$ such that $\Imag\langle v,S_\mathcal{V}v\rangle<0$. Thus, by Condition \eqref{eq:ragout} from Theorem \ref{thm:bern}, the operator $(S_{\mathcal{V,L}}+iV)$ cannot be dissipative for any choice of $\mathcal{L}$ or $V$.

iii) Assume now that there exists an $\varepsilon> 0$ such that $\Imag\langle v,S_{\mathcal{V}}v\rangle=\Imag\langle v,S^*v\rangle\geq\varepsilon\|v\|^2$ for all $v\in\mathcal{V}$. If $\mathcal{L}=0$, \eqref{eq:vondenleyen} clearly holds with $\|\mathcal{L}\|=0$. Now, let $\mathcal{L}\neq 0$. Again, by Condition \eqref{eq:ragout} of Theorem \ref{thm:bern}, the operator $S_{\mathcal{V,L}}+i\frac{\|\mathcal{L}\|^2}{4\varepsilon}$ is dissipative if and only if
\begin{equation} \label{eq:farage}
 \Imag\langle v,S^*v\rangle\geq\frac{1}{4}\left\|\left(\frac{\|\mathcal{L}\|^2}{4\varepsilon}\right)^{-1/2}\mathcal{L}v\right\|^2
\end{equation}
for all $v\in\mathcal{V}$. Since for all $v\in\mathcal{V}$ we may estimate
$$ \frac{1}{4}\left\|\left(\frac{\|\mathcal{L}\|^2}{4\varepsilon}\right)^{-1/2}\mathcal{L}v\right\|^2= \frac{4\varepsilon}{4\|\mathcal{L}\|^2}\|\mathcal{L}v\|^2\leq \varepsilon\|v\|^2\leq\Imag\langle v,S^*v\rangle\:,$$
this proves that \eqref{eq:farage} is satisfied. Hence the operator $S_{\mathcal{V,L}}+i\frac{\|\mathcal{L}\|^2}{4\varepsilon}$ is dissipative, which is equivalent to
$$\Imag\langle \psi,S_{\mathcal{V,L}}\psi\rangle\geq -
\frac{\|\mathcal{L}\|^2}{4\varepsilon}\|\psi\|^2 $$
for all $\psi\in\mathcal{D}(S_{\mathcal{V,L}})$. This finishes the proof.
\end{proof}
\begin{example}[Schr\"odinger operator on the half-line] \normalfont Let $\mathcal{H}=L^2(\R^+)$ and consider the closed symmetric operator $S$ given by:
\begin{align*}
S:\qquad\mathcal{D}(S)=\{f\in H^2(\R^+): f(0)=f'(0)=0\},\quad f\mapsto -f''\:.
\end{align*}
Its adjoint is given by
\begin{align*}
S^*:\qquad\mathcal{D}(S^*)= H^2(\R^+),\quad f\mapsto -f''\:,
\end{align*}
where in both cases, $f''$ denotes the second weak derivative of $f$. Since for any $f\in\mathcal{D}(S^*)$ we have
\begin{equation*}
\Imag\langle f,S^*f\rangle=-\Imag\left(\int_0^\infty \overline{f(x)}f''(x)dx\right)=\Imag \left(\overline{f(0)}f'(0)\right)\:,
\end{equation*}
from which it can be easily shown that all maximally dissipative extensions of $S$ are parametrized by the boundary condition
\begin{align*}
S_h:\qquad\mathcal{D}(S_h)&=\{f\in H^2(\R^+): f'(0)=h f(0)\}\\
f&\mapsto -f''\:,
\end{align*}
where $\Imag(h)\geq 0$. Since $S$ is symmetric, we may choose $$\mathcal{D}(S^*)//\mathcal{D}(S)=\ker(S^*+i)\dot{+}\ker(S^*-i)\:.$$ Now pick $\eta_h\in\mathcal{D}(S^*)//\mathcal{D}(S)$ such that $\eta'_h(0)=h$ and $\eta_h(0)=1$, which means that $\mathcal{D}(S_h)=\mathcal{D}(S)\dot{+}\spann\{\eta_h\}$ with the understanding that $h=\infty$ corresponds to Dirichlet boundary conditions at the origin. This implies that
\begin{equation} \label{eq:labyrinth}
\Imag \langle \eta_h,S^*\eta_h\rangle=\Imag h\:,
\end{equation}
where we introduce the convention that $\Imag(\infty)=0$ since $S_\infty$ is selfadjoint. By Theorem \ref{thm:bern}, Condition \eqref{eq:ragout}, we get that for $h=\infty$ the only linear map $\mathcal{L}$ that describes a dissipative extension $S_{\mathcal{V_\infty,L}}$ is given by $\mathcal{L}\equiv 0$, which corresponds to a proper dissipative extension. Here, $\mathcal{V}_\infty:=\spann\{\eta_\infty\}$. Hence, we will not treat this case anymore from now on.
Now, for $h\neq\infty$, the map $\mathcal{L}$ from $\mathcal{V}=\spann\{\eta_h\}$ has to be of the form $\mathcal{L}\eta_h=k$ for some $k\in\mathcal{H}$. Thus, any $f\in\mathcal{D}(S_h)$ can be written as $f=(f-f(0)\eta_h)+f(0)\eta_h$, where $(f-f(0)\eta_h)\in\mathcal{D}(S)$. This means that the operator $S_{\mathcal{V,L}}$ is given by
\begin{align*} S_{\mathcal{V,L}}:\qquad\mathcal{D}(S_{\mathcal{V,L}})&=\mathcal{D}(S_h)\\
S_{\mathcal{V,L}}f&=-f''+f(0)k\:.
\end{align*}
Since $S_{\mathcal{V,L}}$ only depends on our choice of $h\in\C$ and $k\in\mathcal{H}$, let us use these two parameters to label $S_{\mathcal{V,L}}=S_{h,k}$. Let us now consider a two different bounded dissipative perturbations:
\begin{itemize}
\item Let us start with a rank-one perturbation of the form $V=\alpha|\varphi\rangle\langle\varphi|$, where $\alpha>0$ and $\|\varphi\|=1$. Since $\ran\, V=\ran\, V^{1/2}=\spann\{\varphi\}$, the first condition of Theorem \ref{thm:bern} yields that $k\in\spann\{\varphi\}$. Moreover, on $\spann\{\varphi\}$, the operator $V^{-1/2}$ is given by $\varphi\mapsto \alpha^{-1/2}\varphi$. Thus, the second condition of Theorem \ref{thm:bern} reads as
\begin{equation}
\frac{1}{4}\|\alpha^{-1/2}\lambda\varphi\|^2\leq \Imag h\:\:\Leftrightarrow\:\: |\lambda|^2\leq 4\alpha\Imag h\:,
\end{equation}  
where we have parametrized $k=\lambda\varphi$.
Thus, all (maximally) dissipative extensions of the operator
\begin{align*}
A:\quad\mathcal{D}(A)&=\{f\in H^2(\R^+):f(0)=f'(0)=0\}\\
f&\mapsto -f''+i\alpha \langle \varphi,f\rangle\varphi 
\end{align*}
are given by the family of operators $A_{h,\lambda}$, where $|\lambda|^2\leq 4\alpha\Imag h$:
\begin{align*}
A_{h,\lambda}:\qquad\mathcal{D}(A_{h,\lambda})&=\{f\in H^2(\R^+): f'(0)=h f(0)\}\\
f&\mapsto -f''+f(0)\lambda\varphi+i\alpha\varphi\langle\varphi,f\rangle\:.
\end{align*}
\item Now, let $V$ be the multiplication operator by an a.e.\ non-negative function $V(x)\in L^\infty(\R^+)$. Moreover for any function $h\in L^2(\R^+)$, let $\mathcal{E}_h:=\{x:h(x)\neq 0\}$, which is defined up to a set of Lebesgue measure zero. Clearly, $\overline{\ran\,V}=L^2(\mathcal{E}_V)$. Hence, the first condition of Theorem \ref{thm:bern} yields the requirement that $\mathcal{E}_k\subset\mathcal{E}_V$ up to a set of Lebesgue measure zero. Next, $k\in\mathcal{D}(V^{-1/2})$ implies that $k$ has to be such that
\begin{equation*}
\int_{\mathcal{E}_V}\frac{|k(x)|^2}{V(x)}dx<\infty\:.
\end{equation*}
Lastly, the second condition of Theorem \ref{thm:bern} reads as
\begin{equation*}
\int_{\mathcal{E}_V}\frac{|k(x)|^2}{V(x)}dx\leq 4\Imag h\:.
\end{equation*}
Thus, all (maximally) dissipative extensions of the operator
\begin{align*}
A:\quad\mathcal{D}(A)&=\{f\in H^2(\R^+):f(0)=f'(0)=0\}\\
(Af)(x)&= -f''(x)+iV(x)f(x) 
\end{align*}
are given by the family of operators $A_{h,k}$, where $k\in\mathcal{H}$ such that ${\mathcal{E}_k}\subset\mathcal{E}_V$ (up to a set of Lebesgue measure zero) and
\begin{equation*}
\int_{\mathcal{E}_V}\frac{|k(x)|^2}{V(x)}dx\leq 4\Imag h\:.
\end{equation*}
They are given by:
\begin{align*}
A_{h,k}:\qquad\mathcal{D}(A_{h,\lambda})&=\{f\in H^2(\R^+): f'(0)=h f(0)\}\\
(A_{h,k}f)(x)&= -f''(x)+f(0)k(x)+iV(x)f(x)\:.
\end{align*}
\end{itemize}

\end{example}
\subsection*{Acknowledgements} The main part if this research was carried out during the author's PhD studies at the University of Kent in Canterbury, UK. It is thus a pleasure to acknowledge the support and guidance of his supervisors Ian Wood and Sergey Naboko. Also, he would like to thank them for carefully reviewing this manuscript and making useful suggestions. Moreover, he is indebted to the UK Engineering and Physical Sciences Research Council (Doctoral Training Grant Ref.\ EP/K50306X/1) and the School of Mathematics, Statistics and Actuarial Science at the University of Kent for a PhD studentship.
He is also very grateful to the referees for helpful remarks and valuable suggestions, in particular for pointing out that the proofs of Lemma \ref{lemma:singulariter} and Corollary \ref{coro:ausmachen} cover more than just the essentially selfadjoint case and also for suggesting the example given in Remark \ref{rem:referee}. 

\end{document}